\documentclass[11pt,reqno]{amsart}

\usepackage{amssymb}
\usepackage{amsmath}
\usepackage[latin1]{inputenc}
\usepackage{dsfont}
\usepackage{hyperref}
\usepackage{fullpage}
\usepackage{pdfsync}

\newtheorem{theorem}{Theorem}[section]
\newtheorem{lemma}[theorem]{Lemma}
\newtheorem{proposition}[theorem]{Proposition}
\newtheorem{corollary}[theorem]{Corollary}

\newtheorem{remark}[theorem]{Remark}

\begin{document}
\title{On consecutive values of random completely multiplicative functions}
\author{Joseph Najnudel}
\date{\today}
\maketitle

\begin{abstract}
In this article, we study the behavior of consecutive values of random completely multiplicative functions $(X_n)_{n \geq 1}$ whose values are i.i.d. at primes. We prove that for $X_2$ uniform on the unit circle, or uniform on the set of roots of unity of a given order, and for fixed $k \geq 1$, $X_{n+1}, \dots, X_{n+k}$ are independent if $n$ is large enough. Moreover, with the same assumption, we prove the almost sure convergence of the empirical measure $N^{-1} \sum_{n=1}^N \delta_{(X_{n+1}, \dots, X_{n+k})}$ when $N$ goes to infinity, with an estimate of the rate of convergence. At the end of the paper, we also show that for any probability distribution on the unit circle followed by $X_2$,  the empirical measure converges almost surely when $k=1$.

\end{abstract}
\section{Introduction}
Many arithmetic functions of interest are multiplicative, i.e. their value at $mn$ is the product of their values at $m$ and $n$, for all coprime integers
$m, n \geq 1$. For example, it is the case for the M\"obius function, which is defined by  $\mu(n) = 0$ if $n \geq 1$ is divisible by the square of at least one prime number, 
$\mu(n) = 1$ if $n$ is the product of an even number of distinct primes, and $\mu(n) = -1$ if $n$ is the product of an odd number of distinct primes. 
Similarly, Dirichlet characters are multiplicative, as well as the Liouville function, which is equal to $(-1)^k$ on integers with $k$ prime factors, counted with multiplicity: these functions 
are even completely multiplicative, which means that their value at $mn$ is the product of their values at $m$ and $n$ for all integers $m, n \geq 1$. 
 The behavior of the M\"obius and the Liouville functions is far from being known with complete accuracy, even if partial results have been proven. This difficulty can be encoded by the corresponding Dirichlet series, which involve the Riemann zeta function. For example, the partial sum, up to $x$, 
of the M\"obius function in known to be 
negligible with respect to $x$, and it is conjectured to be negligible with respect to $x^{r}$ for all $r > 1/2$: the first statement can quite easily be proven to be equivalent to the prime number theorem, whereas the second is equivalent to the Riemann hypothesis. 

It has been noticed that the same bound $x^r$ for all $r > 1/2$ is obtained if we take the partial sums of i.i.d., bounded and centered random variables. This suggests the naive idea to compare the M\"obius function on square-free integers with i.i.d. random variables on $\{-1,1\}$. However, a major difference between the two situations is that in the random case, we lose the multiplicativity of the function. 
A less naive randomized version of M\"obius functions can be obtained as follows: one takes i.i.d. uniform random variables on $\{-1,1\}$ on prime numbers, $0$ on prime powers of order larger than or equal to $2$, and one completes the definition by multiplicativity. 

In \cite{W}, Wintner considers a completely multiplicative function with i.i.d. values at primes, uniform on $\{-1,1\}$ (which corresponds to a randomized version of the Liouville function rather than the M\"obius function), and proves that we have almost surely the same bound $x^{r}$ ($r > 1/2$) for the partial sums, as for the sums of i.i.d. random variables, or for the partial sums of M\"obius function if the Riemann hypothesis is true. The estimate in \cite{W} has been refined by Hal\'asz in \cite{Hal}, and then by Lau, Tenenbaum and 
Wu in \cite{LTW}. Some lower bounds can also be deduced from moment estimates by Harper \cite{Harper}. 
In order to get more general results, it can be useful to consider complex-valued random multiplicative functions. For example, it has been proven by Bohr and Jessen \cite{BJ} that for $\sigma >1/2$, 
the law of $\zeta(\sigma + iTU)$, for $U$ uniformly distributed on $[0,1]$, tends to a limiting random variable when $T$ goes to infinity. 
This limiting random variable can be written as $\sum_{n \geq 1} X_n n^{-\sigma}$, when $(X_n)_{n \geq 1}$ is a random completely multiplicative function such that $(X_p)_{p \in \mathcal{P}}$ are i.i.d. uniform on the unit circle, $\mathcal{P}$ denoting, as in all the sequel of the present paper, the set of prime numbers. The fact that the series just above converges is a direct consequence (by partial summation) of the analog of the result of Wintner for the partial sums of $(X_n)_{n \geq 1}$: one can prove that almost surely, $\sum_{n \leq x} X_n = o(x^r)$ for $r > 1/2$.  

This discussion shows that it is often much less difficult to prove accurate results for random multiplicative function than for the arithmetic functions which are usually considered. In some informal sense, the arithmetic difficulties are diluted into the randomization, which is much simpler to deal with. 

In the present paper, we study another example of results which are stronger and less difficult to prove in the random setting than in the deterministic one. 
The example we detail in this article is motivated by the following question, initially posed in the deterministic setting: for $k \geq 1$, what can we say about the distribution of the $k$-tuples
$(\mu(n+1), \dots, \mu(n+k))$, or $(\lambda(n+1), \dots, \lambda(n+k))$, where $\mu$ and $\lambda$ are the M\"obius and the Liouville functions, 
$n$ varies from $1$ to $N$, $N$ tends to infinity? This question is only very partially solved. One knows (it is essentially a consequence of the prime number theorem), that for $k = 1$, the proportion of integers such that $\lambda$ is equal to $1$ or $-1$ tends to $1/2$. For the M\"obius function, the  limiting proportions are $3/\pi^2$ for $1$ or $-1$ and $1 - (6/\pi^2)$ for $0$. 
It has been proven by Hildebrand \cite{Hil} that for $k=3$, the eight possible values 
of $(\lambda(n+1), \lambda(n+2), \lambda(n+3))$ appears infinitely often. This result has been improved by Matom\"aki, Radziwi\l \l \, and Tao \cite{MRT}, who prove that these eight values  
 appear with a positive lower density: in other words, for all $(\epsilon_1, \epsilon_2, \epsilon_3) \in \{-1,1\}^3$,
 $$\underset{N \rightarrow \infty}{\lim \inf}  \frac{1}{N} \sum_{n=1}^N 
 \mathds{1}_{\lambda(n+1) = \epsilon_1, 
 \lambda(n+2) = \epsilon_2,
 \lambda(n+3) = \epsilon_3} > 0.$$
 The similar result is proven for the nine possible values of $(\mu(n+1), \mu(n+2))$. A conjecture by Chowla \cite{Ch} states that for all $k \geq 1$, each possible pattern of $(\lambda(n+1), \dots, \lambda(n+k))$ appears with asymptotic density $2^{-k}$.  This conjecture is still open, however, partial results have been recently proven, in particular in papers by Tao and Ter\"av\"ainen (\cite{TT1}, \cite{TT2}, \cite{TT3}). 

In the present paper, we prove results similar to this conjecture for random completely multiplicative functions $(X_n)_{n \geq 1}$. The random functions we will consider take i.i.d. values on the unit circle on prime numbers. Their distribution is then 
entirely determined by the distribution of $X_2$. The two particular cases we will study in the largest part of the paper are the following: $X_2$ is uniform on the unit circle $\mathbb{U}$, and $X_2$ is uniform on the set $\mathbb{U}_q$ of $q$-th roots of unity, for $q \geq 2$. In this case, we will show the following results: for all $k \geq 1$, and for all $n \geq 1$ large enough depending on $k$, the variables $X_{n+1}, \dots, X_{n+k}$ are independent, and exactly i.i.d. uniform on the unit circle if $X_2$ is uniform. Moreover, the empirical distribution 
$$\frac{1}{N} \sum_{n=1}^N \delta_{(X_{n+1}, \dots, X_{n+k})}$$
tends almost surely to the uniform distribution on $\mathbb{U}^k$ if $X_2$ is uniform on $\mathbb{U}$, and to the uniform distribution on $\mathbb{U}_q^k$ if $X_2$ is uniform on $\mathbb{U}_q$. 
In particular, the analog of Chowla's conjecture holds almost surely in the case where $X_2$ is uniform on $\{-1,1\}$. We have also an estimate on the speed of convergence of the empirical measure: in the case of the uniform distribution 
on $\mathbb{U}_q$, each of the 
$q^k$ possible patterns for $(X_{n+1}, \dots, X_{n+k})$ almost surely occurs with a proportion $q^{-k} + O(N^{-t})$ for $n$ running between $1$ and $N$, for all $t < 1/2$. We have a similar result in the uniform case, if the test functions we consider are sufficiently smooth. 
It would be interesting to have similar results when the distribution of $X_2$ on the unit circle is not specified. For $k \geq 2$, we are unfortunately not able to show similar results, but we nevertheless can prove that the empirical distribution of $X_n$ almost surely converges to a limiting distribution for any distribution of $X_2$ on the unit circle. We specify this distribution, which is always uniform on $\mathbb{U}$ or uniform on $\mathbb{U}_q$ for some $q \geq 1$, and in the latter case, we give an estimate of the rate of convergence. This rate corresponds to a negative power of $\log N$, which is much slower than what we obtain when  $X_2$ is uniform on $\mathbb{U}_q$. 

The techniques we use in our proofs are  elementary in general, mixing classical tools in probability theory and number theory. However, a part of our arguments need to use deep  results on diophantine equations, in order to bound the number and the size of their solutions. 

The sequel of the paper is organized as follows. In Sections \ref{uniform} and \ref{uniformq}, we study the law 
of $(X_{n+1}, \dots, X_{n+k})$ for $n$ large depending on $k$, first in the case where $X_2$ is uniform on $\mathbb{U}$, then in the case where $X_2$ is uniform on $\mathbb{U}_q$. In Section \ref{uniform2}, we study the empirical measure of 
$(X_{n+1}, \dots, X_{n+k})$ in the case of $X_2$ uniform on $\mathbb{U}$. In the proof of the convergence of this empirical measure, we need to estimate the second moment of sums of the form $\sum_{n=N'+1}^N \prod_{j=1}^k X_{n+j}^{m_j}$. The problem of estimating moments of order different from two for such sums is discussed in Section \ref{moments}. The proof of convergence of empirical measure in the case of uniform variables on $\mathbb{U}_q$ is given in Section \ref{uniform2q}. 
Finally, we consider the case of a general distribution for $X_2$ in Section \ref{general}.

\section{Independence in the uniform case} \label{uniform}

In this section, we suppose that $(X_p)_{p \in 
\mathcal{P}}$ are i.i.d. uniform random variables on the unit circle. By convenience, we will extend our multiplicative function to positive rational numbers by  setting $X_{p/q} := X_p/X_q$: the result is independent
of the choice of $p$ and $q$, and we have $X_r X_s = X_{rs}$ for all rationals $r, s > 0$. Moreover, 
$X_r$  is uniform on the unit circle for all positive rational $r \neq 1$. 
In this section, we will show that for fixed $k \geq 1$, 
$(X_{n+1},\dots, X_{n+k})$ are independent if $n$ is sufficiently large. The following result gives 
a criterion for such independence: 
\begin{lemma} \label{1.1}
 For all $n, k \geq 1$, the variables $(X_{n+1},\dots, X_{n+k})$ are independent   if and only if 
 $\log(n+1), \dots, \log(n+k)$ are linearly independent on $\mathbb{Q}$.
\end{lemma}
\begin{proof}
 Since the variables $(X_{n+1}, \dots, X_{n+k})$ are uniform on the unit circle, they are independent if and only
 if 
 $$\mathbb{E}[X_{n+1}^{m_1} \dots X_{n+k}^{m_k}] = 0$$
 for all $(m_1, \dots, m_k) \in \mathbb{Z}^k \backslash \{(0,0, \dots, 0)\}$.
 This equality is equivalent to 
 $$ \mathbb{E}[X_{(n+1)^{m_1} \dots (n+k)^{m_k}}] = 0,$$
 i.e. $$(n+1)^{m_1} \dots (n+1)^{m_k} \neq 1$$
 or 
 \begin{equation} m_1 \log (n+1) + \dots + m_k \log (n+k) \neq 0. \label{1}
 \end{equation}
\end{proof}
We then get the following result: 
\begin{proposition} \label{1.2}
The variables $(X_{n+1}, \dots, X_{n+k})$ are i.i.d. as soon as $n \geq (100k)^{k+1}  $. In particular, for $k$ fixed, this property 
is true for all but finitely many $n$. 
\end{proposition}
\begin{remark}
The same result is proven in \cite{T}, Theorem 3. (i),  
with an asymptotically better bound, namely $n \geq e^{ck \log \log (k+2)/ \log (k+1)}$ where $c >0$ is a constant. 
 However, their proof uses a deep result by Shorey \cite{Sh} on linear forms in the logarithms of algebraic numbers, involving technical tools by Gelfond and Baker, whereas our proof is elementary. Moreover, the constant $c$ involved in \cite{T} is not given, even if it is explicitly computable. 
\end{remark}
\begin{proof}
 Let us assume that we have a linear dependence \eqref{1} between $\log (n+1), \dots, \log(n+k)$: necessarily $k \geq 2$.
 Moreover, the integers $n+j$ for which $m_j \neq 0$ cannot be divisible by a prime  larger than $k$: otherwise 
 this factor remains in the product 
 $\prod_{\ell=1}^k (n+\ell)^{m_j}$ since none of the $n+\ell$ for $\ell \neq j$ can be divisible by $p$, and then the product cannot be equal to $1$. 
 We can rewrite the dependence as follows: 
 $$\log(n+j) = \sum_{\ell \in A} r_{\ell} \log (n+\ell),$$
 for a subset $A$ of $\{1, \dots, k\} \backslash \{j\}$ and for $R := (r_{\ell})_{\ell \in A} \in \mathbb{Q}^A$. Let us assume that the cardinality $|A|$ is as small as possible. Taking the decomposition in prime factors, we get for all 
 $p \in \mathcal{P}$, 
 $$v_p(n+j) = \sum_{\ell \in A} v_p(n+\ell) r_{\ell},$$
 where $v_p$ denotes the exponent of $p$ in the prime factorization. 
 If $M := (v_p(n+\ell))_{p \in \mathcal{P}, \ell \in A}$, $V := (v_p(n+j))_{p \in \mathcal{P}}$, then we can write 
 these equalities in a matricial way $V = M R$. The minimality of $|A|$ ensures that the matrix $M$ has rank $|A|$. 
Moreover, since all the prime factors of 
 $(n+\ell)_{\ell \in A}$ are smaller than $k$, 
 all the rows of $M$ indexed by prime numbers larger than $k$ are identically zero, and then the rank $|A|$ of $M$ is at most 
 $\pi(k)$, the number of primes smaller than or equal to $k$. 
Moreover, we can extract a subset $\mathcal{Q}$ of $\mathcal{P}$ of cardinality $|A|$  such that the restriction 
 $M^{(\mathcal{Q})}$ of $M$ to the rows with indices in $\mathcal{Q}$ is invertible. 
 We have with obvious notation: $V^{(\mathcal{Q})} = M^{(\mathcal{Q})} R$, and then by Cramer's rule, 
 the entries of $R$ can be written as the quotients of determinants of matrices obtained from $M^{(\mathcal{Q})}$
 by replacing one column by $V^{(\mathcal{Q})}$, by the determinant of  $M^{(\mathcal{Q})}$. 
 All the entries involved in these matrices are $p$-adic valuations of integers smaller than or equal to $n+k$, so 
 they are at most $\log (n+k)/ \log 2$. By Hadamard inequality, the absolute value of the
 determinants are smaller than or equal to 
 $([\log (n+k)/\log(2)]^{|A|}) |A|^{|A|/2}$. Since $|A| \leq \pi(k)$, we deduce, after multiplying by $\det 
 (M^{(\mathcal{Q})})$,
 that there exists a linear dependence between $\log (n+1), \dots, \log (n+k)$ involving only 
 integers of absolute value at most $D := [\sqrt{\pi(k)} \log (n+k)/\log2]^{\pi(k)}$: 
 let us keep the notation of \eqref{1} for this dependence. 
 Let $q$ be the smallest nonnegative integer such that $\sum_{j=1}^k j^q m_j \neq 0$: from the fact that the 
 Vandermonde matrices are invertible, one deduces that $q \leq k-1$. Using the fact that 
 $$\left| \log(n+j) - \left( \log n + \sum_{r = 1}^{q} (-1)^{r-1} \frac{j^r}{r n^r} \right) \right|
 \leq \frac{j^{q+1}}{(q+1) n^{q+1}},$$
 we deduce, by writing the dependence above: 
 $$ \left| \sum_{j=1}^k j^q m_j \right|  \frac{1}{q n^q}  \leq \sum_{j=1}^k \frac{|m_j| j^{q+1}}{(q+1) n^{q+1}}$$
 if $q \geq 1$ and 
 $$ \left| \sum_{j=1}^k  m_j \right|  \log n  \leq \sum_{j=1}^k \frac{|m_j| j}{n}$$
 if $q = 0$. 
 Since the first factor in the left-hand side of these inequalities is a non-zero integer, it is at least $1$. 
 From the bounds we have on the $m_j$'s, we deduce 
 $$  \frac{1}{q n^q}  \leq \frac{ D k^{q+2}}{(q+1) n^{q+1}}$$
 for $q \geq 1$ and 
 $$ \log n \leq  \frac{ D k^2}{n}.$$
 for $q = 0$. Hence 
 $$1 \leq  \left(\frac{q}{q+1} \vee \frac{1}{\log n} \right) \frac{ D k^{q+2}}{n} \leq  \frac{ D k^{q+2}}{n}$$
if $n \geq 3$, which implies, since $q \leq k-1$, 
$$n \leq D k^{k+1} \leq  [\sqrt{\pi(k)} \log (n+k)/\log2]^{\pi(k)} k^{k+1}.$$
If $n \geq k \vee 3$, we deduce 
$$ 2 n \leq 2 [\sqrt{\pi(k)} \log (2n)/\log2]^{\pi(k)} k^{k+1},$$
i.e.
$$ \frac{2n}{[\log (2n)]^{\pi(k)}} \leq 2 [\sqrt{\pi(k)}/\log2]^{\pi(k)} k^{k+1}
.$$
Now, one has obviously $\pi(k) \leq 2k/3$ for all $k \geq 2$, and then 
 $\sqrt{\pi(k)}/\log 2  \leq  \sqrt{2 k}$ for all integers $k \geq 2$, and more accurately, it is known that 
 $(\pi(k) \log k )/k$, which tends to $1$ at infinity by the prime number theorem, reaches its maximum at $k = 113$: this fact is
 in particular an immediate consequence of  \cite{RS}, Corollary 2, equation (3.7). 
 Hence, 
  $$\frac{2n}{[\log (2n)]^{\pi(k)}}
  \leq 2 (2 k)^{c k /\log k} k^{k+1}$$
  where 
  $$c = \frac{1}{2} \, \frac{\pi(113) \log 113}{113} \leq 0.63$$
  and then 
  $$\frac{2n}{[\log (2n)]^{\pi(k)}} 
  \leq 2 ( 2^{0.63 k/ \log 2}) k^{0.63 k/\log k} k^{k+1} \leq 2 e^{1.26 k} k^{k+1}
  \leq (e^{1.26} k)^{k+1} \leq (3.6 k)^{k+1}.$$

Let us assume that $2n \geq (100k)^{k+1}$. The function
$x \mapsto x/\log^{\pi(k)}(x)$ is increasing
for $x \geq e^{\pi(k)}$. Moreover, by
studying the function $x \mapsto 
\log \log (100 x) / \log (x+1)$ for $x \geq 2$, we check that $\log (100k) \leq (k+1)^{1.52}$ for all $k \geq 2$. Hence, since $\pi(k) \leq \pi(k+1)$, 
$$\frac{2n}{[\log (2n)]^{\pi(k)}}
\geq \frac{(100k)^{k+1}}{
((k+1)\log (100k))^{\pi(k)}}
\geq \frac{(100k)^{k+1}}{(k+1)^{2.52 \pi(k+1)}} 
$$ $$\geq \frac{(100k)^{k+1}}{(k+1)^{(2.52)(1.26) (k+1)/\log(k+1)}} 
\geq (100 k e^{-3.18})^{k+1} \geq (4 k)^{k+1},$$
which contradicts the previous inequality.

 Hence, 
 $$n \leq 2n \leq (100k)^{k+1},$$
and this bound is of course also available for $n \leq k \vee 3$.
\end{proof}
This result implies that theoretically, for fixed $k$, one can find all the values of $n$ such that $(X_{n+1}, \dots, X_{n+k})$ are not independent by brute force computation. In practice, the bound we have obtained is far from optimal, and is too poor to be directly useable except for very small values of $k$, for which a more careful reasoning can solve the problem directly. Here is an example for $k=5$: 
\begin{proposition}
For $n \geq 1$, the variables $(X_{n+1},X_{n+2}, X_{n+3}, X_{n+4}, X_{n+5})$ are independent except if $n \in \{1,2,3,4,5,7\}$. 
\end{proposition}
\begin{proof}
If $\prod_{j=1}^5 (n+j)^{m_j} = 1$ with integers $m_1, \dots m_5$ not all equal to zero,  
then $m_j = 0$ as soon as $n+j$ has a prime factor larger than or equal to $5$: otherwise, this prime factor cannot be cancelled by the factors $(n+k)^{m_k}$ for $k \neq j$. Hence, the values of $n+j$ such that $m_j \neq 0$ have only prime factors $2$ and $3$, and at most one of them has both factors since it should then be divisible by $6$. Moreover, if $n \geq 4$, there can be at most one power of $2$ and one power of $3$ among $n+1, \dots, n+5$. 
One deduces that dependence is only possible if among $n+1, \dots, n+5$, there are three numbers, respectively of the form
$2^k, 3^\ell, 2^r.3^s$, for integers $k, \ell, r, s > 0$. The quotient between two of these integers is between $1/2$ and $2$ since we here assume $n \geq 4$.  Hence, $2^k \geq  2^r.3^s /2 \geq 
2^r$ and then $k \geq r$. Similarly, 
$3^{\ell} \geq 2^r.3^s /2 \geq 3^s$, which implies $\ell \geq s$. The numbers $2^k$ and $2^r.3^s$ are then both divisible by 
$2^r$; since they differ by at most $4$, 
$r \leq 2$. The numbers $3^{\ell}$ and 
  $2^r.3^s$ are both divisible by $3^s$, and then $s \leq 1$. Therefore, 
  $2^r.3^s \leq 12$ and $n \leq 11$. 
  If $9 \leq n \leq 11$, the only possible values of $n+j$ such that $m_j$ can be different from zero are $12$ and $16$, which are multiplicatively independent. 
  If $n = 8$, the only possible values are $9$ and $12$, which are also independent, if $n = 6$, the values to consider are $8$ and $9$. 
  The only remaining values of $n$ are $1, 2, 3, 4, 5, 7$, which are exceptions since 
  $$8^4 \cdot 9^3 \cdot 12^{-6} =  6^{-6} \cdot 8^2 \cdot 9^3 = 4^3 \cdot 8^{-2} = 3^2 \cdot 4 \cdot 6^{-2} = 2^2 \cdot 4^{-1} = 1.$$
  \end{proof}
The results above give an upper bound, for fixed $k$, of the maximal value of 
$n$ such that $(X_{n+1}, \dots, X_{n+k})$ are not independent. By considering two consecutive squares and their geometric mean, whose logarithms are linearly dependent, one deduces the lower bound 
$ ([k/2]-1)^2 - 1 \geq (k-1)(k-5)/4$ for the maximal $n$. 
As written in a note by Dubickas \cite{D}, this bound can be improved to  a quantity equivalent to $(k/4)^3$, 
by considering the identity: 
$$(n^3 - 3n - 2)(n^3 - 3n +2) n^3 
= (n^3 - 4n)(n^3 - n)^2.$$
In \cite{D}, as an improvement of a result of \cite{T}, it is also shown that for all $\epsilon > 0$, the lower bound
$e^{\log^2 k /[(4 + \epsilon) \log \log k]}$ occurs for infinitely many values of $k$. 

A computer search 
gives, for $k$ between $3$ and $13$, and 
$n \leq 1000$, the following largest values   for which we do not have independent variables: 1, 5, 7, 14, 23, 24, 47, 71, 71, 71, 239. For example, if $k = 13$ and $n = 239$, the five integers $240, 243, 245, 250, 252$ have only the four prime factors $2, 3, 5, 7$, so we necessarily have a dependence, namely: 
$$240^{65} \cdot 243^{31} \cdot
245^{55} \cdot 250^{-40} \cdot
252^{-110} = 1.$$
 It would remain to check if there are dependences for $n > 1000$. 
\section{Independence in the case of roots of unity}  \label{uniformq}
We now suppose that 
$(X_p)_{p \geq 1}$ are i.i.d., uniform on the set of $q$-th roots of unity, $q \geq 1$ being a fixed integer. If $q = 2$, we get symmetric Bernoulli random variables.
For all integers $s \geq 2$, we will denote
by $\mu_{s,q}$ the largest divisor $d$ of 
$q$ such that $s$ is a $d$-th power.
The analog of
Lemma \ref{1.1} in the present setting is the following: 
 \begin{lemma}
 For $n, k \geq 1$, the variables 
 $(X_{n+1}, \dots, X_{n+k})$ are all
 uniform on the set of $q$-th roots of unity if and only if $\mu_{n+j,q} = 1$ for all $j$ between $1$ and $k$. They are 
  independent if and only if 
 the only $k$-tuple $(m_1, \dots, m_k)$, 
 $0 \leq m_j < q/ \mu_{n+j,q}$ such that 
 $$\forall p \in \mathcal{P}, \sum_{j=1}^k m_j v_p (n+j) 
 \equiv 0 \, (\operatorname{mod. } q)$$
 is $(0,0,\dots,0)$. 
 \end{lemma}
 \begin{proof}
 For any $s \geq 2$, $\ell \in \mathbb{Z}$,  we have 
 $$\mathbb{E}[X_s^\ell]  = 
 \prod_{p \in \mathcal{P}} \mathbb{E} [X_p^
 {\ell v_p(s)}],$$
 which is equal to $1$ if 
 $\ell v_p(s)$ is divisible by $q$ for all $p \in \mathcal{P}$, and to $0$ otherwise. 
 The condition giving $1$ is equivalent 
 to the fact that $\ell$ is a multiple 
 of $ q/(\operatorname{gcd}(q, (v_p(s))_{p \in  \mathcal{P}}))$, which is $q/\mu_{s,q}$.
 Hence, $X_s$ is a uniform $(q/\mu_{s,q})$-th root of unity, which implies the first part of the proposition. 
 
 The variables $(X_{n+1}, \dots, X_{n+k})$ are independent if and only if for all $m_1, \dots, m_k \in \mathbb{Z}$, 
 $$ \mathbb{E} \left[ \prod_{j=1}^k 
 X_{n+j}^{m_j} \right] 
  = \prod_{j=1}^k \mathbb{E}[ X_{n+j}^{m_j} ].$$
Since $X_{n+j}$ is a  uniform $(q/\mu_{n+j,q})$-th root of unity, both sides of the equality depend only on the values of $m_j$ modulo $q/\mu_{n+j,q}$ for $1 \leq j \leq k$. This implies that we can assume, without loss of generality, that $0 \leq m_j < q/\mu_{n+j,q}$ for all $j$. If all the $m_j$'s are zero, both sides are obviously equal to $1$. Otherwise, the right-hand side is equal to zero, and then we have independence if and only if it is also the case of the left-hand side, i.e. for all 
$(m_1, \dots, m_k) \neq (0,0, \dots, 0)$, 
$0  \leq m_j < q/\mu_{n+j,q}$, 
$$\mathbb{E} \left[ \prod_{j=1}^k 
 X_{n+j}^{m_j} \right]  
 = \mathbb{E} \left[ \prod_{p \in \mathcal{P}} X_p^{\sum_{1 \leq j \leq k} 
 m_j v_p(n+j)} \right] 
  = \prod_{p \in \mathcal{P}}
  \mathbb{E} \left[  X_p^{\sum_{1 \leq j \leq k} 
 m_j v_p(n+j)} \right]  = 0,$$
 which is true if and only if 
$$\exists p \in \mathcal{P}, \sum_{j=1}^k m_j v_p (n+j) 
 \not\equiv 0 \, (\operatorname{mod. } q).$$
 \end{proof}
 We then have the following result, similar to Proposition \ref{1.2}:
 \begin{proposition}
 For fixed $k, q \geq 1$, there exists an explicitely computable $n_0(k,q)$ such that $(X_{n+1}, \dots, X_{n+k})$ are independent as soon as $n \geq n_0(k,q)$. 
 \end{proposition}
The bound $n_0(k,q)$ can be deduced from  bounds on the solutions of certain diophantine equations which are available in the literature: we do not take care of its precise value, which is anyway far too large to be of any use if we want to find in practice the values of $n$ such that $(X_{n+1}, \dots, X_{n+k})$  are not independent. 
 \begin{proof}
 For each value of $n \geq 1$ such that 
  $(X_{n+1}, \dots, X_{n+k})$ are dependent, 
  there exist $0 \leq m_j < q/\mu_{n+j,q}$, not all zero, such that 
  $$\forall p \in \mathcal{P}, \sum_{j=1}^k m_j v_p (n+j) 
 \equiv 0 \, (\operatorname{mod. } q).$$
 There are finitely many choices, depending only on $k$ and $q$, for 
 the $k$-tuples $(\mu_{n+j,q})_{1 \leq j \leq k}$ and 
 $(m_j)_{1 \leq j \leq k}$, so it is sufficient to show that the values of $n$ corresponding to each 
choice of $k$-tuples is bounded by an explicitely computable quantity.
 At least two of the $m_j$'s are non-zero: otherwise $m_j v_p(n+j)$ is divisible by 
 $q$ for all $p \in \mathcal{P}$, $j$ being the unique index such that $m_j \neq 0$, and then $m_j$ is divisible by $q/\mu_{n+j,q}$: this contradicts the inequality 
 $0 < m_j < q/\mu_{n+j,q}$.
 
 On the other hand, if $p$ is a prime larger than $k$, at most one of the terms 
 $m_j v_p(n+j)$ is non-zero, and then 
 all the terms are divisible by $q$, since 
 it is the case for their sum. 
  
  We deduce that $n+j$ is the product of a power of order $\rho_j := q/\operatorname{gcd}(m_j,q)$ 
  and a number $A_j$ whose prime factors are all
  smaller than $k$. Moreover, one can assume 
  that $A_j$ is "$\rho_j$-th power free", i.e. that all its $p$-adic valuations
  are strictly smaller than $\rho_j$.
  Hence there exist 
  $$A_j \leq \prod_{p \in \mathcal{P}, 
  p \leq k} p^{\rho_j - 1} \leq (k!)^q$$
 and an integer $B_j \geq 1$ such that 
 $n+j = A_j B_j^{\rho_j}$.
 The value of the exponents $\rho_j$ are
 fixed by the $m_j$'s, and at least two 
 of them are strictly larger than $1$, since at least two of the $m_j$'s are non-zero.  Let us first assume that there exist distinct $j$ and $j'$ such that $\rho_j \geq 2$ and $\rho_{j'} \geq 3$. 
One finds an explicitly computable bound on $n$ in this case as soon as we find an explicitly computable bound for the solutions of each diophantine equation in $x$ and $z$:  
 $$A z^{\rho_j} - A' x{^{\rho_{j'}}} = d$$
 for each $A, A', d$ such that $1 \leq A, A' \leq (k!)^q$ and $-k < d < k$, $d \neq 0$. 
These equations can be rewritten as:  $y^{\rho_j} = f(x)$, where $y = Az$ and 
$$f(x) = A^{\rho_j - 1} (A'x^{\rho_{j'}} + d).$$
This polynomial has all simple roots (the 
$\rho_{j'}$-th roots of $-d/A'$) and then 
at least two of them; it has at least three if $\rho_j = 2$ since $\rho_{j'}$ is supposed to be at least $3$ in this case. 
By a result of Baker \cite{B}, all the solutions are bounded by an explicitly computable quantity, which gives the desired result (the same result with an ineffective bound was already proven by Siegel). 

In remains to deal with the case where $\rho_j = 2$ for all $j$ such that $m_j \neq 0$. In this case, $q$ is even and $m_j$ is divisible by $q/2$, which implies that $m_j = q/2$ when $m_j \neq 0$. By looking at the prime  factors larger than $k$, one deduces that for all $j$ such that $m_j \neq 0$, 
$n+j$ is a square times a product of distinct primes smaller than or equal to $k$. If at least three of the $m_j$'s are non-zero, it then suffices to find an explicitly computable bound for the solutions of each system of diophantine equations: 
$$ B y^2  = A t^2 + d_1,  C z^2 = A t^2 + d_2$$
for $1 \leq A, B, C  \leq  k!$ squarefree, $-k < d_1, d_2 < k$, $d_1, d_2, d_1 - d_2 \neq 0$.
From these equations, we deduce, for $x = BC yz$:  
$$x^2 = BC (At^2 + d_1)(At^2 + d_2).$$
The four roots of the right-hand side are the square roots of $-d_1/A$ and $-d_2/A$, which are all distinct since $d_1 \neq d_2$, $d_1 \neq 0$, $d_2 \neq 0$. Again by Baker's result, one deduces that the solutions are explicitly bounded, which then gives an explicit bound for $n$. 
 
 The remaining case is when exactly two of the $m_j$'s are non-zero, with $\rho_j = 2$, and then  $m_j = q/2$. 
 The dependence modulo $q$ then means that 
 $(n+j)(n+j')$ is a square for distinct $j, j'$ between $1$ and $k$, which implies that 
 $(n+j)/g$ and $(n+j')/g$ are both squares 
 where $g = \operatorname{gcd}(n+j, n+j')$. 
 These squares have difference smaller than $k$, which implies that they are smaller than $k^2$. Moreover, $g$ divides $|j-j'| \leq k$, and then $g \leq k$, which gives $n \leq k^3$. 
 \end{proof}
 Here, we explicitly solve a particular case: 
 \begin{proposition}
 For $q = 2$, $(X_{n+1}, \dots, X_{n+5})$ are independent for all $n \geq 2$ and not for $n=1$. 
 \end{proposition}
\begin{proof}
A dependence means that there exists a product 
of distinct non-square integers among $n+1, \dots, n+5$ which is a square. 
For a prime $p \geq 5$, at most one $p$-adic valuation is non-zero, which implies that all the $p$-adic valuations are even. 
Hence, the factors involved in the product are all squares multiplied by $2, 3$ or $6$. Since they differ by at most $4$, they cannot be in the same of the three "categories", which implies, since the product is a square, that there exist three numbers, respectively of the form $2x^2$, $3y^2$, $6z^2$, in the interval between $n+1$ and $n+5$. Now, Hajdu and Pint\'er  \cite{HP} have determined all the triples of distinct integers in intervals of length at most 12 whose product is a square. For length $5$, the only positive triple is $(2,3,6)$, which implies that the only dependence in the present setting is $X_2 X_3 X_6 = 1$. 
\end{proof}
\begin{remark}
The list given in \cite{HP} shows that for $q = 2$, there are dependences for quite large values of $n$ as soon as $k \geq 6$. For example, 
we have $X_{240} X_{243} X_{245} = 1$ for $k =6$ and $X_{10082}X_{10086}X_{10092} = 1$ for $k = 11$. 
\end{remark}

\section{Convergence of the empirical measure in the uniform case}  \label{uniform2}
In this section, $(X_p)_{p \in \mathcal{P}}$ are uniform on the unit circle, and $k \geq 1$ is a fixed integer. For $N \geq 1$, we consider the empirical measure of the $N$ first $k$-tuples: 
$$\mu_{k,N} := \frac{1}{N}\sum_{n=1}^N 
\delta_{(X_{n+1}, \dots, X_{n+k})}.$$ 
It is reasonable to expect that $\mu_{k,N}$ tends to the uniform distribution on $\mathbb{U}^k$, which is the common distribution of 
$(X_{n+1}, \dots X_{n+k})$ for all but finitely many values of $n$.  
In order to prove this result, we will estimate the second moment of the Fourier transform of $\mu_{k,N}$, given by 
$$\hat{\mu}_{k,N}(m_1, \dots, m_k) 
= \int_{\mathbb{U}^k} \prod_{j=1}^k 
z_j^{m_j} d\mu_{k,N} (z_1, \dots, z_k).$$ 
\begin{proposition} \label{momentorder2}
Let $m_1, \dots, m_k$ be integers, not all 
equal to zero. Then, for all $N > N' \geq 0$, 
$$\mathbb{E} \left[ 
\left|\sum_{n=N'+1}^N \prod_{j=1}^k X_{n+j}^{m_j} \right|^2 \right] 
\leq k (N-N') $$
and there exists $C_{m_1, \dots, m_k} \geq 0$, independent of $N$ and $N'$, such that  
$$N-N' \leq \mathbb{E} \left[ 
\left|\sum_{n=N'+1}^N \prod_{j=1}^k X_{n+j}^{m_j} \right|^2 \right] 
 \leq N - N' + C_{m_1, \dots, m_k}.$$
Moreover, under the same assumption, 
$$\mathbb{E} \left[ |\hat{\mu}_{k,N}(m_1, \dots, m_N)|^2
\right] \leq \frac{k}{N},$$
$$\frac{1}{N} \leq \mathbb{E} \left[ |\hat{\mu}_{k,N}(m_1, \dots, m_N)|^2
\right] \leq \frac{1}{N} + \frac{C_{m_1, \dots, m_k}}{N^2}.$$
Finally, for $k \in \{1,2\}$, one can take
$C_{m_1}$ or $C_{m_1, m_2}$ equal to $0$, 
and for $k = 3$, one can take 
$C_{m_1, m_2,m_3} = 2$ if $(m_1,m_2,m_3)$ is proportional to $(2,1,-4)$ and $C_{m_1,m_2,m_3} = 0$ otherwise. 
\end{proposition}
\begin{proof}
We have, using the completely multiplicative extension of $X_r$ to all $r \in \mathbb{Q}_+^*$: 
$$\mathbb{E} \left[ 
\left|\sum_{n=N'+1}^N \prod_{j=1}^k X_{n+j}^{m_j} \right|^2 \right] 
 = \sum_{N' < n_1, n_2 \leq N}
 \mathbb{E} \left[X_{\prod_{j=1}^{k}
  (n_1 + j)^{m_j} / (n_2 + j)^{m_j} } \right],$$
  and then the left-hand side is equal to the number of couples $(n_1, n_2)$ 
  in $\{N'+1, \dots, N\}^2$ such that 
  \begin{equation}\prod_{j=1}^k (n_1 + j)^{m_j}
 =  \prod_{j=1}^k (n_2 + j)^{m_j}.
 \label{n1n2}
 \end{equation}
 The number of trivial solutions $n_1 = n_2 $ of this equation is equal to $N - N'$, which gives a lower bound on the second moment we have to estimate. 
On the other hand,  the derivative of the rational fraction $\prod_{j=1}^k (X + j)^{m_j}$ can be written as the product of $\prod_{j=1}^k (X + j)^{m_j - 1}$, which is strictly positive on $\mathbb{R}_+$, by the 
polynomial 
$$Q(X) = \prod_{j=1}^k (X+j)  \left[
\sum_{j=1}^k \frac{m_j}{X + j} \right].$$
 The polyomial $Q$ has degree at most $k-1$ and is non-zero, since $(m_1, \dots, m_k)
 \neq (0, \dots, 0)$ and then $\prod_{j=1}^k
 (X+j)^{m_j}$ is non-constant. 
 We deduce that $Q$ has at most $k-1$ zeros, and then on $\mathbb{R}_+$, $\prod_{j=1}^k (X+j)^{m_j}$ is strictly monotonic on each of at most $k$ intervals of $\mathbb{R}_+$, whose bounds are $0$, the positive zeros of $Q$ and $+\infty$. Hence, for each choice of $n_1$, there are at most $k$ values of $n_2$ 
 satisfying \eqref{n1n2}, i.e. at most one in each interval, which gives the upper bound $k(N-N')$ for the moment we are estimating. 

Moreover, since $\prod_{j=1}^k (X+j)^{m_j}$ is strictly monotonic on an interval 
of the form $[A, \infty)$ for some $A > 0$, we deduce that for any non-trivial 
solution $(n_1,n_2)$ of  \eqref{n1n2}, the minimum of $n_1$ and $n_2$ is at most $A$. 
Hence, there are finitely many possibilities for the common value of the two sides of 
\eqref{n1n2}, and for each of these values, at most $k$ possibilities for $n_1$ and for  $n_2$. Hence, for fixed $(m_1, \dots, m_k)$, the total number of non-trivial solutions of \eqref{n1n2} is finite, which gives the bound $N-N'+ C_{m_1, \dots, m_k}$ of the proposition. 

The statement involving the empirical measure is deduced by taking $N' = 0$ and by dividing everything by $N^2$. 

The claim for $k \leq 3$ is an immediate consequence of the following statement we will prove now: the only integers $n_1 > n_2 \geq 1$, $(m_1,m_2,m_3) \neq (0,0,0)$, 
such that \begin{equation}
(n_1+1)^{m_1} (n_1+2)^{m_2} 
(n_1+3)^{m_3}  = (n_2+1)^{m_1} (n_2+2)^{m_2} 
(n_2+3)^{m_3} \label{n1n22}
\end{equation}
are $n_1 = 7$, $n_2 = 2$, $(m_1,m_2,m_3)$ proportional to $(2,1,-4)$, which corresponds  to  the equality: $$
 8^2 \cdot 9 \cdot 10^{-4} = 3^2 \cdot 4 \cdot 5^{-4}.$$
 If $m_1, m_2, m_3$ have the same sign and are not all zero, $(n+1)^{m_1}(n+2)^{m_2}
 (n+3)^{m_3}$ is strictly monotonic in $n \geq 1$, and then we cannot get a solution of  \eqref{n1n22} with $n_1 > n_2$. 
 By changing all the signs if necessary, we may assume that one of the integers $m_1, m_2,m_3$ is strictly negative and the others are nonnegative. For $n \geq 1$, 
 the fraction obtained by writing 
 $(n+1)^{m_1}(n+2)^{m_2}
 (n+3)^{m_3}$ can only be simplified by prime 
 factors dividing two of the integers $n+1, n+2, n+3$, and then only by a power of $2$. 
If $m_2 < 0$ and then $m_1, m_3 \geq 0$, the numerator and the denominator have different parity, and then the fraction is irreducible for all $n$: we do not get any solution of \eqref{n1n22} in this case. 
Otherwise, $m_1$ or $m_3$ is strictly negative. If $(n_1, n_2)$ solves \eqref{n1n22}, let us define $s := 1$ and $j := n_2 + 1$ if $m_1 < 0$, and $s := -1$ and $j := n_2 + 3$ if $m_3 < 0$.   
The denominators of the two fractions corresponding to the two sides of \eqref{n1n22} are respectively a power of 
$j$ and the same power of $n_1 + 2 - s$: if \eqref{n1n22} is satisfied, these denominators should differ only by a power of $2$, since the fractions can be only simplified by such a power. 
Hence, $n_1 + 2 - s = 2^{\ell} j$ for some $\ell \geq 0$, and by looking at the numerators of the fractions, we deduce that there exists $r \geq 0$ such that 
$$2^r (j+s)^{m_2} (j+2s)^{m_{2 + s}}
=  (2^{\ell}j+s)^{m_2} (2^{\ell}j+2s)^{m_{2 + s}}.$$
If $\ell \geq 2$, the ratios 
$(2^{\ell}j+s)/(j+s)$ and 
$(2^{\ell}j+2s)/(j+2s)$ are at least 
$(4\cdot2 + 2)/(2 + 2) = 5/2$ since $j \geq n_2 +1 \geq 2$ and $|2s| \leq 2$, and then 
the ratio between the right-hand side
and the left-hand side of the previous equality is at least $(5/2)^{m_{2 + s} + m_2} 2^{-r}$, 
which gives 
$$ 2^r \geq (5/2)^{m_{2 + s} + m_2}.$$
On the  other hand, the $2$-adic valution
of the right-hand side is $m_{2+s}$ since 
$2^{\ell} j + 2s \equiv 2$ modulo 4, whereas the valuation of the left-hand side is at least $r$, which gives 
$$ 2^r \leq 2^{m_{2+s}}.$$
We then get a contradiction for $\ell \geq 2$, except in the case $m_{2+s} = m_2 = 0$, 
where we already know that there is no solution of \eqref{n1n22}. 
If $\ell = 1$, we get 
$$2^r (j+s)^{m_2} (j+2s)^{m_{2 + s}}
=  (2j+s)^{m_2} (2j+2s)^{m_{2 + s}}.$$
In this case, the prime factors of $2j+s$, which are odd ($|s| = 1$), should divide $j+s$ or $j+2s$, then $2j+2s$ or $2j+4s$, and finally $s$ or $3s$. Hence, $2j+s$ is a power of $3$. 
Similarly, the odd factors of $j+2s$, and then of $2j + 4s$, should divide $2j+s$ or $2j+2s$, and then $s$ or $3s$: $2j+4s$ is the product of a power of $2$ and a power of $3$. 
If we write $2j+s = 3^a$, $2j + 4s = 2^b 3^c$, we must have $|3^a - 2^b 3^c| = 3$.
If $a \leq 1$, we have $2j + s \leq 3$. If $s = 1$, we get $n_2 + 1 = j \leq 1$, and if $s = -1$, we get $n_2 + 3 = j \leq 2$, which is impossible. 
If $a \geq 2$, $3^a$ is divisible by $9$, and then $2^b 3^c$ is congruent to $3$ or $6$ modulo $9$, which implies $c = 1$, and then $|3^{a-1} - 2^b| = 1$. Now, by induction, one proves that the order of $2$ modulo $3^{a-1}$ is equal to $2.3^{a-2}$ (i.e. $2$ is a primitive root modulo the powers of $3$). This result is classical, and can be deduced, for example, 
from Rosen \cite{Rosen}, Theorem 8.9. For sake of completeness, we give a proof here. 
The result is easy to check be direct computation for $a = 2$ and $a = 3$. Let us assume that it is true for all values until $a \geq 3$. The order of $2$ modulo $3^a$ is a multiple 
of the order of $2$ modulo $3^{a-1}$, and then a multiple of  $2.3^{a-2}$ by assumption. On the other hand, it is a divisor of $2.3^{a-1}$
by Euler's theorem. Hence, it is either $2.3^{a-2}$ or $2.3^{a-1}$. Moreover, since $2.3^{a-3}$ is assumed to be 
the order of $2$ modulo $3^{a-2}$ but strictly smaller than the order of $2$ modulo $3^{a-1}$, we have 
$$2^{2.3^{a-3}} = 1 + u.3^{a-2}$$
where $u$ is not divisible by $3$. Raising to the cube, we deduce 
$$2^{2.3^{a-2}} = 1 + 3 u.3^{a-2} +  3 u^2.3^{2a-4} + u^3 3^{3a-6} = 1 + v.3^{a-1}$$
where 
$$ v = u + u^2.3^{a-2} + u^3. 3^{2a-5}$$
is not divisible by $3$ (recall that $a \geq 3$ here). Hence, the order of $2$ modulo $3^a$ is not $2.3^{a-2}$: it can only be $2.3^{a-1}$, which proves by induction 
that $2$ is a primitive root of $3^{a-1}$ for all $a \geq 2$. 
Now, in the present situation, the order of $2$ modulo $3^{a-1}$, i.e. $2.3^{a-2}$, should divide $2b$, since $2^b \equiv \pm 1$ modulo $3^{a-1}$, 
and then  $b \geq 3^{a-2}$ ($b = 0$ is not possible) which implies 
$2^{3^{a-2}} \leq 3^{a-1} + 1$, i.e. $a \in \{2,3\}$. 

If $a = 2$ and $s = 1$, we get $2j+1 = 9$, 
$j= 4$, and then $n_1 = 7$, $n_2 = 3$. 
We should solve $4^{m_1} 5^{m_2} 6^{m_3} = 
8^{m_1} 9^{m_2} 10^{m_3}$. Taking the $3$-adic valuation gives $m_3 = 2 m_2$, taking the $5$-adic valuation gives $m_3 = m_2$, and then $m_2 = m_3 = 0$, which implies $m_1 = 0$. 

If $a = 2$ and $s = -1$, we get $2j - 1 = 9$, $j=5$, $n_1 = 7$, $n_2 = 2$, which gives the equation $3^{m_1} 4^{m_2} 5^{m_3} = 
8^{m_1} 9^{m_2} 10^{m_3}$. 
Taking the $2$-adic valuation gives $2m_2 = 3m_1 + m_3$, taking the $3$-adic valuation gives $m_1 = 2 m_2$, and then $(m_1,m_2,m_3)$ should be proportional to $(2,1,-4)$: in this case, we get one of the  solutions already mentioned. 

If $a = 3$, $2^b$ should be $8$ or $10$, and then $b = 3$, $2j+ s = 27$, $2j + 4s = 24$, $j = 14$, $s = -1$, $n_1 = 25$, $n_2 = 11$. We have to solve $12^{m_1} 13^{m_2} 14^{m_3} = 26^{m_1} 27^{m_2} 28^{m_3}$. 
Taking the $3$-adic valuation gives $m_1 = 3m_2$, taking the $13$-adic valuation gives $m_1 = m_2$, and then $m_1 = m_2 = m_3 = 0$.

\end{proof}
\begin{corollary}
For all $(m_1, \dots, m_k) \in \mathbb{Z}^k$, $\hat{\mu}_{k,N}(m_1, \dots, m_k)$ converges 
in $L^2$, and then in probability, to $\mathds{1}_{m_1 = \dots = m_k = 0}$, i.e. to the corresponding Fourier coefficient of the uniform distribution $\mu_{k}$ on $\mathbb{U}^k$. In other words, $\mu_{k,N}$ converges weakly in probability to $\mu_{k}$. 
\end{corollary}
In this setting, we also have a strong law of large numbers, with an estimate of the 
rate of convergence, for sufficiently smooth test functions. Before stating the corresponding result, we will show the following lemma, which will be useful: 
\begin{lemma} \label{lemmaLLN}
Let $\epsilon > \delta \geq 0$, $C > 0$, and let $(A_n)_{n \geq 0}$ be a sequence of 
random variables such that $A_0 = 0$ and
for all $N > N' \geq 0$, 
$$\mathbb{E}[|A_{N} - A_{N'}|^2] 
\leq C (N-N')N^{2 \delta}.$$
Then, almost surely, $A_N = O(N^{1/2 + \epsilon})$: more precisely, we have for $M > 0$,   
$$\mathbb{P} \left(\sup_{N \geq 1} 
|A_N|/(N^{1/2 + \epsilon}) \geq M 
\right) \leq K_{\epsilon, \delta} C M^{-2},$$
where $K_{\epsilon, \delta} > 0$ depends 
only on $\delta$ and $\epsilon$. 
\end{lemma}
\begin{proof}
For $\ell, q \geq 0$, $M > 0$ and $\epsilon' := (\delta + \epsilon)/2 
\in (\delta, \epsilon)$, we have:
\begin{align*}\mathbb{P} \left(|A_{(2\ell + 1).2^q} - 
A_{(2\ell).2^q}| \geq M [(2\ell + 1).2^q]^{1/2 + \epsilon'} \right)
& \leq M^{-2} [(2\ell + 1).2^q]^{-1 -2 \epsilon'}
\mathbb{E} \left[|A_{(2\ell + 1).2^q} - 
A_{(2\ell).2^q}|^2\right]
\\ & \leq M^{-2}.C.2^q.[(2\ell + 1).2^q]^{2 \delta -1 -2 \epsilon'}
\\ & \leq M^{-2}.C.2^{-2q (\epsilon' - \delta)}
(2\ell + 1)^{-1 - 2 (\epsilon' - \delta)}.
\end{align*}
Since $\epsilon' > \delta$, we deduce that the probability that 
\begin{equation}
|A_{(2\ell + 1).2^q} - 
A_{(2\ell).2^q}| < M [(2\ell + 1).2^q]^{1/2 + \epsilon'} \label{xcvb}
\end{equation}
for all $\ell, q \geq 0$ is at least $1 - D CM^{-2}$, where $D$ depends only on $\epsilon'$ and $\delta$, and then only on $\delta$ and $\epsilon$. 
Now, if \eqref{xcvb} occurs for all $\ell, q \geq 0$, if 
we take the binary expansion  $N = \sum_{j=0}^{\infty} \delta_j
 2^j$
with $\delta_j \in \{0,1\}$, and if $N_r 
= \sum_{j=r}^{\infty} \delta_j 2^j$ for all $r \geq 0$, then we get $|A_{N_r} - A_{N_{r+1}}|
= 0$ if $\delta_r = 0$, and 
\begin{align*} |A_{N_r} - A_{N_{r+1}}|
& = |A_{2^r (2 (N_{r+1}/2^{r+1}) + 1)}
 - A_{2^r (2 N_{r+1}/2^{r+1})}|
\\ &   \leq M[2^r (2( N_{r+1}/2^{r+1}) + 1)]^{1/2 + \epsilon'} 
= M (N_r)^{1/2 + \epsilon'} \leq M N^{1/2 + \epsilon'}
\end{align*}
if $\delta_r = 1$. Adding these inequalities from $r = 0$ to $\infty$, we deduce that $|A_N| \leq M \mu(N) N^{1/2
+ \epsilon'}$, where $\mu(N)$ is the number of $1$'s in the binary expansion of $N$. Hence, 
$$|A_N| \leq M \left(1 + (\log N/\log 2)
\right) N^{1/2 + \epsilon'} 
 < B M N^{1/2 + \epsilon}  ,$$
where $B > 0$ depends only on $\epsilon'$ and $\epsilon$ (recall that $\epsilon > \epsilon'$), and then only on $\delta$ and $\epsilon$. 
We then have, for $M' := BM$: 
$$\mathbb{P} \left(\exists N \geq 1, 
|A_N| \geq  M' N^{1/2 + \epsilon} \right)
\leq D C M^{-2} = D C B^{2} (M')^{-2},$$
which gives the desired result after replacing $M'$ by $M$. 
\end{proof}
From this lemma, we deduce the following:   
\begin{proposition}
Almost surely, $\mu_{k,N}$ weakly converges to $\mu_k$. More precisely, the following holds with probability one: for all $u > k/2$, for all continuous functions $f$ from 
$\mathbb{U}^k$ to $\mathbb{C}$ such that 
$$\sum_{m \in \mathbb{Z}^k} |\hat{f}(m)| 
\, ||m||^{u} < \infty,$$ 
$|| \cdot ||$ denoting any norm on 
$\mathbb{R}^k$, 
and for all $\epsilon > 0$, 
$$\int_{\mathbb{U}^k} f d \mu_{k,N} 
= \int_{\mathbb{U}^k} f d \mu_{k}  
+ O(N^{-1/2 + \epsilon}).$$
\end{proposition}
\begin{remark}
By Cauchy-Schwarz inequality, we have 
$$\sum_{m  \in \mathbb{Z}^k} 
|\hat{f}(m)| (1+||m||)^{u}
\leq \left(\sum_{m  \in \mathbb{Z}^k} 
|\hat{f}(m)|^2 (1 + ||m||)^{4u} \right)^{1/2}
\left(\sum_{m  \in \mathbb{Z}^k} 
(1+||m||)^{- 2u} \right)^{1/2},$$
which implies that the assumption on $f$ given in the proposition is satisfied for all $f$ in the Sobolev space $H^s$ as soon 
as $s > k$.

Unfortunately, the proposition does not apply if $f$ is a product of indicators of arcs. The weak convergence implies that  
$$\int_{\mathbb{U}^k} f d\mu_{k,N} \underset{N \rightarrow \infty}{\longrightarrow} 
 \int_{\mathbb{U}^k} f d\mu_{k} $$
 even in this case, but we don't know at which rate this convergence occurs. 

\end{remark}
\begin{proof}
 From Proposition \ref{momentorder2}, and 
Lemma \ref{lemmaLLN} applied to $\epsilon > 0$, $\delta = 0$ and  
$$A_N := N \hat{\mu}_{k,N} (m),$$
we get, for all $m \in \mathbb{Z}^k \backslash \{0\}$, 
$M  > 0$,
$$\mathbb{P} \left(\sup_{N \geq 1} 
| \hat{\mu}_{k,N} (m)|/N^{-1/2 + \epsilon} \geq M 
\right) \leq k K_{\epsilon, 0}  M^{-2}$$
For fixed $u > k/2$, we apply this estimate to $M = ||m||^u$ and get 
$$\mathbb{P} \left(\sup_{N \geq 1} 
| \hat{\mu}_{k,N} (m)|/N^{-1/2 + \epsilon} \geq ||m||^{u}  \right) \leq 
k K_{\epsilon, 0} ||m||^{- 2u}.$$
Since $- 2u < -k$, we deduce, by Borel-Cantelli lemma, that almost surely, 
$$\sup_{N \geq 1} 
| \hat{\mu}_{k,N} (m)|/(N^{-1/2 + \epsilon}
||m||^u) \leq 1$$
for all but finitely many $m \in \mathbb{Z}^k \backslash \{0\}$. Therefore,  almost surely, 
$$\sup_{m \in \mathbb{Z}^k \backslash \{0\}} 
\sup_{N \geq 1} 
| \hat{\mu}_{k,N} (m)|/(N^{-1/2 + \epsilon}
||m||^u) < \infty$$
i.e. 
$$\hat{\mu}_{k,N} (m)
= O(N^{-1/2 + \epsilon} ||m||^u)$$
for $m \in \mathbb{Z}^k \backslash \{0\}$, $N \geq 1$. 
 Almost surely, this estimates simultaneously occurs for all rationals $u > k/2$ and $\epsilon > 0$  (with a random implicit constant in $O$, depending on $u$ and $\epsilon$) and then for all reals $u > k/2$ and $\epsilon > 0$.

Let us now assume that this almost sure property holds, let us fix $u > k/2$, $\epsilon > 0$, and let $f$ be a function 
satisfying the assumptions of the proposition.  Since the Fourier coefficients of $f$ are summable (i.e. $f$ is in the Wiener algebra of $\mathbb{U}^k$), the corresponding Fourier series converges uniformly to a function which is necessarily equal to $f$, since it has the same Fourier coefficients. We can then 
write:
$$f(z_1, \dots, z_k) = 
\sum_{m_1, \dots, m_k \in \mathbb{Z}}
\hat{f}(m_1, \dots, m_k) \prod_{j=1}^k
z_j^{m_j},$$
which implies 
$$\int_{\mathbb{U}^k} f d\mu_{k,N} 
= \sum_{m \in \mathbb{Z}^k}
\hat{f}(m)
\hat{\mu}_{k,N} (m) 
= \int_{\mathbb{U}^k} f d\mu_{k} 
+ \sum_{m \in \mathbb{Z}^k \backslash \{0\}}
\hat{f}(m)
\hat{\mu}_{k,N} (m).$$
By assumption, the last sum is dominated by 
$$N^{-1/2 + \epsilon} \sum_{m \in \mathbb{Z}^k}
|\hat{f}(m)| ||m||^u,$$
which is finite by the assumptions made in the proposition,  and then $O(N^{-1/2 + \epsilon})$.

\end{proof}
\section{Moments of order different from two}  \label{moments}
Since we have a law of large numbers on
 $\mu_{k,N}$, with rate of decay of order $N^{-1/2+ \epsilon}$, it is natural to look if we have a central limit theorem. 
 In order to do that, a possibility consists in studying moments of sums in $n$ of products of variables from $X_{n+1}$ to $X_{n+k}$.
 For the sums $\sum_{n=1}^N X_n$, we do not have convergence to a non-zero  Gaussian random variable after normalization by $1/\sqrt{N}$. Indeed, the second moment of the absolute value of the renormalized sum $\frac{1}{\sqrt{N}} \sum_{n=1}^N X_n$
 is equal to $1$, so if this variable converges to a non-zero complex Gaussian variable, we need to have the convergence of 
 $$\mathbb{E} \left[\frac{1}{\sqrt{N}}\left|\sum_{n=1}^N
  X_n \right| \right]$$
  towards a non-zero constant.  In \cite{HNR}, Harper, Nikeghbali and Radziwi\l \l \, prove that the quantity just above decays at most like $(\log \log N)^{-3 + o(1)}$ when $N$ goes to infinity, whereas
 a conjecture by Helson \cite{H} states that it tends to $0$. The order of magnitude of the left-hand side has later been found by Harper in \cite{Harper}: it is $(\log \log N)^{-1/4}$, which 
 in particular proves Helson's conjecture. 
 
 On the other hand, an equivalent of the moments of $\frac{1}{\sqrt{N}}\left|\sum_{n=1}^N
  X_n \right|$ of even integer order are computed in \cite{HNR} and \cite{HL}, and they are not bounded with respect to $N$: the moment of order $2p$ is equivalent to an explicit constant times $(\log N)^{(p-1)^2}$. The order of magnitude of the moments of any positive order, not necessarily integer, is given by Harper in \cite{Harper} and \cite{Harper2}. 
  
 In the case of sums different from $\sum_{n=1}^N X_n$, the moment computations involve arithmetic problems of different nature: here, we look in some detail the case of 
 the sum $\sum_{n=1}^N X_n X_{n+1}$. In this case, the fact that consecutive, and then necessarily coprime integers are involved gives more independence than when we study the 
 sum  $\sum_{n=1}^N X_n$. In particular, it seems reasonable to expect that  $\sum_{n=1}^N X_n X_{n+1}$ satisfies the same central limit theorem as the sum of i.i.d. uniform variables 
 on the unit circle, and that this fact can be proven by moment computations. The convergence of the second moment is obvious, and we will now show that the convergence of 
 the fourth moment also occurs. 
 We start with the following result: 
 \begin{proposition} 
 We have
 $$\mathbb{E}\left[ \left|\sum_{n=1}^N
 X_n X_{n+1} \right|^4 \right] 
  = 2N^2 - N + 8 \mathcal{N}(N) 
  +   4 \mathcal{N}_{=} (N),$$
  where $\mathcal{N}(N)$ (resp. $\mathcal{N}_{=}(N)$) is the number of solutions of the diophantine equation
  $a(a+1)d(d+1) = b(b+1)c(c+1)$ such that
  the integers $a, b, c, d$ satisfy 
  $0 < a  < b < c <d \leq N$ (resp. $0 < a < b = c < d \leq N$). Moreover, for all $\epsilon > 0$, there exists  
  $C_{\epsilon} > 0$, independent of $N$, such that for all $N \geq 8$, 
  $$ N/2 \leq 8 \mathcal{N}(N) 
  +   4 \mathcal{N}_{=} (N) \leq C_{\epsilon} N^{3/2 + \epsilon}.$$
  Hence, 
  $$\mathbb{E}\left[ \left|\sum_{n=1}^N
 X_n X_{n+1} \right|^4 \right] 
  = 2N^2 + O_{\epsilon} (N^{3/2 + \epsilon}).$$
 \end{proposition}
 \begin{proof}
 Expanding the fourth moment, we immediately obtain that 
 it is equal to the total number of solutions of the previous diophantine equation, with 
 $a, b, c, d \in \{1,2, \dots, N\}$. 
 One has $2N^2 - N$ trivial solutions: $N(N-1)$ for which $a = c\neq b = d$, $N(N-1)$ for which $a = b \neq c = d$, $N$ for which $a = b = c = d$. It remains to count the number of non-trivial solutions. 
 Such a solution has a minimal element among $a, b, c, d$. This element is unique: if two minimal elements are on the same side, then necessarily $a = b = c=d$, if two minimal elements are on different sides, then the other elements should be equal, which also gives a trivial solution. 
 Dividing the number of solutions by four, we can assume that $a$ is the unique smallest integer, which implies that $d$ is the largest one. For $b = c$, we get $\mathcal{N}_= (N)$ solutions, and for 
 $b \neq c$, we get $ 2 \mathcal{N}_= (N)$
 solutions, the factor $2$ coming from the possible exchange between $b$ and $c$.
 
The lower bound $N/2$ comes from the solutions $(1,3,3,8)$ and $(1,2,5,9)$ for $8 \leq N \leq 24$, and from the solutions 
of the form $(n,2n+1,3n,6n+2)$ for $N \geq 25$. 

Let us now prove the upper bound. We start by slightly simplifying the equation by introducing the odd integers
$A = 2a + 1$, $B = 2b+1$, $C = 2c + 1$, $D = 2d + 1$, which should satisfy:
$$(A^2 - 1)(D^2-1) = (B^2 - 1)(C^2 -1).$$
If $A, B, C, D$ are large, then $AD$ and $BC$ should be odd and close to each other. It is then quite natural to introduce
$$\delta := (AD - BC)/2,$$
which is expected to be small with respect to $A, B, C, D$. More precisely, 
since $B$ and $C$ are closer to each other than $A$ and $D$, we need
$$A^2 - 1+ D^2-1 > B^2 - 1 + C^2 - 1,$$
and then $\delta > 0$, 
since
$$A^2 D^2 - B^2 C^2 = A^2 + D^2 - B^2 - C^2 > 0.$$
The last equality, gives, after factorizing the left-hand side and replacing $BC$ by $AD - 2 \delta$: 
$$4 \delta (AD - \delta) = A^2  + D^2 
-  (B-C)^2 - 2AD + 4 \delta,$$
and in particular
$$A^2 - 2(2 \delta +1) AD + D^2 + 4 \delta (\delta +1 ) = (B-C)^2 \geq 0.$$
If we neglect the term $4 \delta (\delta +1 ) $, expected to be small with respect to $AD$, we get the 
positivity of a quadratic form in $A$ and $D$, which gives a restriction on the possible values of the ratio $D/A$. 
More precisely, if we assume $1 < D/A \leq 2 \delta + 2$, we deduce 
  $$AD \left( \frac{1}{2 \delta + 2} 
  - 4 \delta - 2 + 2\delta + 2 \right)  + 4 \delta(\delta +1) \geq 0,$$
  and then 
  $$AD  \leq \frac{4 \delta(\delta + 1)}{ 2\delta - (1/4)} = 2 (\delta +1) 
  \left(1 - \frac{1}{8 \delta} \right)^{-1}
  \leq 2 (\delta + 1) \left( 1 + \frac{1}{7 \delta} \right) \leq 2 \delta + 2 + (4/7), 
  $$
  $AD \leq 2 \delta +1$ since it is an odd integer, and then $BC = AD - 2 \delta \leq 1$, which gives a contradiction. 
  Any solution should then satisfy $D/A > 2\delta +2$. 
 We now discuss in function of the value of $\delta$. 
  For $\delta > \sqrt{N}$, we have necessarily $A < D /( 2\sqrt{N} + 2)
   \leq (2N+1)/(2 \sqrt{N} + 2) = O(\sqrt{N})$, and then $a = O(\sqrt{N})$, 
   and then there are only $O(N^{3/2})$ possibilities for the couple $(a,d)$. Now, $b$ and $c$ should be divisors of $a(a+1)d(d+1) = O(N^4)$, and by the classical divisor bound, we deduce that there are $O(N^{\epsilon})$ possibilities for $(b,c)$ when $a$ and $d$ are chosen. 
   Hence, the number of solutions for $\delta > \sqrt{N}$ is bounded by the estimate we have claimed. 
 
 It remains to bound the number of solutions for  $\delta \leq \sqrt{N}$: we will get a bound for the number of solutions for each value of $\delta$, which will be multiplied by $\sqrt{N}$
 at the end. 
Each solution should satisfy
$$A^2 - 2(2 \delta +1) AD + D^2 + 4 \delta (\delta +1 ) = (B-C)^2,$$
i.e. by writing the quadratic form in $A$ and $D$ as a difference of squares:
$$[D - (2 \delta + 1)A]^2 +  4\delta (\delta +1 )
= 4\delta(\delta + 1) A^2  + (B-C)^2.$$
We know that $D  \geq A (2 \delta +2)$, and then $0 < D - (2 \delta + 1)A \leq 2N+1$, which gives, for each value of $\delta$, 
$O(N)$ possibilities for $D - (2\delta +1)A$. For the moment, let us admit that for each of these possibilities, there are $O(N^{\epsilon})$ choices for $B-C$ and $A$.   Then, for fixed $\delta$, we have 
$O (N^{1+ \epsilon})$ choices for $(D -
(2\delta +1) A, A, B-C)$. For each choice, $B-C, A, D$ are fixed, and then also $BC = AD - 2 \delta$, and finally $B$ and $C$. 
Hence, we have $O(N^{1+\epsilon})$ solutions for each $\delta \leq \sqrt{N}$, and then $O(N^{3/2 + \epsilon})$ solutions by counting all the possible $\delta$. 

The claim we have admitted is a consequence of the following fact we will prove now: for $\epsilon > 0$,  
the number of representations of $M$ in integers by the quadratic form $X^2 + P Y^2$ is $O(M^{\epsilon})$, uniformly in the strictly positive integer $P$. Indeed,
for such a representation, the ideal 
$(X + Y \sqrt{-P})$ should be a divisor of 
$(M)$ in the ring of integers $\mathcal{O}_P$ of $\mathbb{Q}[\sqrt{-P}]$, and each such ideal gives at most $6$ couples $(X,Y)$ representing $M$. 
Indeed, the group of  invertible elements in $\mathcal{O}_P$ has order at most $6$. This fact is classical (see for example Jarvis \cite{Jarvis}, Chapter 6), and can 
be proven as follows: 
 if $\alpha + \beta \sqrt{-P}$ is invertible in $\mathcal{O}_P$ for $\alpha, \beta \in \mathbb{Q}$,
then $\alpha - \beta \sqrt{-P}$ is also invertible in $\mathcal{O}_P$, and  
$$(\alpha + \beta \sqrt{-P} )+ (\alpha - \beta \sqrt{-P} ) = 2 \alpha$$
is an integer since it is in $\mathcal{O}_P$, whereas
$$(\alpha + \beta \sqrt{-P} )(\alpha - \beta \sqrt{-P} ) = \alpha^2 + P \beta^2$$
is an invertible integer, necessarily equal to $1$. Hence, $\alpha + \beta \sqrt{-P} $ is a complex number of modulus $1$, with real part equal to $-1,-1/2, 0, 1/2$ or $1$, i.e.
a fourth root or a sixth root of unity. We now only need to bound the number of divisors of $(M)$ in 
$\mathcal{O}_P$ by $O(M^{\epsilon})$, uniformly in $P$. The number of divisors of $(M)$ is $\prod_{\mathfrak{p}} (v_\mathfrak{p}(M) + 1)$, where we have the prime ideal decomposition 
$$(M) = \prod_{\mathfrak{p}} \mathfrak{p}^
{v_\mathfrak{p}(M)}.$$
Now, by considering the decomposition of  prime numbers as products of ideals, we deduce:  
$$(M) = \prod_{p \in \mathcal{P}, \, p  \operatorname{inert}} (p)^{v_p(M)}
\prod_{p \in \mathcal{P}, \, p  \operatorname{ramified}} \mathfrak{p}_p^{2 v_p(M)}
\prod_{p \in \mathcal{P}, \, p  \operatorname{split}} \mathfrak{p}_p^{ v_p(M)} \overline{\mathfrak{p}_p}^{\, v_p(M)},
 $$
$\mathfrak{p}_p$ denoting an ideal of norm $p$, and then the number of divisors of $(M)$ is
$$\prod_{p \in \mathcal{P}, \, p  \operatorname{inert}} (v_p(M) + 1)
\prod_{p \in \mathcal{P}, \, p  \operatorname{ramified}} (2 v_p(M) + 1)
\prod_{p \in \mathcal{P}, \, p  \operatorname{split}} (v_p(M) + 1)^2
\leq \prod_{p \in \mathcal{P}}
(v_p(M) + 1)^2 = [\tau(M)]^2,
 $$
where $\tau(M)$ is the number of divisors, in the usual sense, 
 of the integer $M$. This gives the desired bound $O(M^{\epsilon})$.  
 \end{proof}
 \begin{remark}
Using the previous proof, one can show the following quite curious property: all the  solutions of $a(a+1) d(d+1) = b(b+1) c (c+1)$ in integers $0 < a < b \leq c < d$ satisfy $d/a > 3 + 2 \sqrt{2}$. Indeed, let us assume the contrary. With the previous  notation, $3 + 2 \sqrt{2} \geq d/a \geq  D/A
 > 2 \delta + 2$, and then $\delta = 1$, which gives $A^2 - 6AD + D^2 + 8 \geq 0$, 
i.e.
 $$(2a + 1)^2 - 6(2a+1)(2d+1) + (2d+1)^2 
  + 8 \geq 0,$$
  $$ 4 (a^2 - 6 ad + d^2) - 8a - 8d + 4 \geq 0,$$
  a contradiction since $1 < d/a \leq 3 + 2 \sqrt{2}$ implies $a^2 - 6ad + d^2 \leq 0$. 
  The bound $3+2 \sqrt{2}$ is sharp, since we have the solutions of the form
   $(u_{2k}, u_{2k+1}, u_{2k+1}, u_{2k+2})$, where 
   $$u_r := \frac{(1+\sqrt{2})^{r} 
  + (1 - \sqrt{2})^{r} - 2}{4}.$$
  
 \end{remark}
 A consequence of the previous proposition corresponds to a bound on all the moments of order $0$ to $4$: 
 \begin{corollary}
 We have, for all $q \in [0,2]$, 
 $$c_q + o(1) \leq \mathbb{E} \left[  \left| \frac{1}{\sqrt{N}} \sum_{n=1}^N X_n X_{n+1} \right|^{2q} \right] \leq C_q + o(1),$$
 where $c_q = 2^{-(q-1)_{-}} \geq 1/2$ and 
 $C_q = 2^{(q-1)_{+}} \leq 2$. 
 \end{corollary}
  \begin{proof}
  H\"older inequality implies that the logarithm of the $2q$-th moment of a nonnegative random variable is a convex function of $q$.
  Now, we have proven that this logarithm is equal to $0$ for $q = 0$ and $q = 1$ and to $\ln 2 + o(1)$ for $q = 2$. 
The corollary can now be deduced from the following fact, easy to check: if $f$ is a convex fonction from $[0,2]$ to $\mathbb{R}$ such that 
$f(0) = f(1) = 0$, $f(2) = 1$, then 
$$f(x) \leq 0 \cdot \mathds{1}_{x \in [0,1)} + (x-1) \mathds{1}_{x \in [1,2]} = (x-1)_+$$
and 
$$f(x) \geq (x - 1)\mathds{1}_{x \in [0,1)} + 0 \cdot \mathds{1}_{x \in [1,2]} = -(x-1)_-.$$
 \end{proof}
 We have proven that the fourth moment of $\left| \frac{1}{\sqrt{N}} \sum_{n=1}^N X_n X_{n+1} \right|$ converges to $2$, which is also the limit of 
 the fourth moment of $\left| \frac{1}{\sqrt{N}} \sum_{n=1}^N Z_n \right|$
 where $(Z_n)_{n \geq 1}$ are i.i.d. random variables, uniform on the unit circle. 
 Unfortunately, we are not able to prove a similar convergence for higher moments, and then we do not know how to prove a central limit theorem. 
  However, the following result holds: 
  \begin{proposition}
 If for all integers $q \geq 1$, the  number of non-trivial solutions
 $(n_1, \dots, n_{2q}) \in \{1, \dots, N\}^{2q}$ of the diophantine equation
$$\prod_{r=1}^q n_r(n_r+1) 
= \prod_{r=1}^q n_{q+r} (n_{q+r} + 1)$$
 is negligible with respect to the number of trivial solutions when $N \rightarrow \infty$ (i.e. $o(N^q)$) then we have 
 $$ \frac{1}{\sqrt{N}} \sum_{n=1}^{N} X_n X_{n+1} \underset{N \rightarrow \infty}{\longrightarrow}  \mathcal{N}_{\mathbb{C}},$$
 where $\mathcal{N}_{\mathbb{C}}$ denotes a standard Gaussian complex variable, i.e. 
 $(\mathcal{N}_1 + i \mathcal{N}_2)/\sqrt{2}$ where $\mathcal{N}_1, \mathcal{N}_2$ are independent standard real Gaussian variables. 
 
 \end{proposition}
\begin{proof} 
If $$Y_N := \frac{1}{\sqrt{N}} \sum_{n=1}^{N} X_n X_{n+1},$$ 
then for integers $q_1, q_2 \geq 0$, 
the moment $\mathbb{E}[Y_N^{q_1} \overline{Y_N}^{\, q_2}]$ is equal to $N^{-(q_1+ q_2)/2}$ times the number of solutions
$(n_1, \dots, n_{q_{1} + q_2}) \in \{1, \dots, N\}^{q_1 + q_2}$ of 
$$\prod_{r=1}^{q_1} n_r(n_r+1) 
= \prod_{r=1}^{q_2} n_{q_1+r} (n_{q_1+r} + 1).$$
If $0 \leq q_1 < q_2$, there are at most $N^{q_1}$ choices for $n_1, \dots, n_{q_1}$, and once these integers are fixed, at most $N^{o(1)}$ choices for $n_{q_1+1}, \dots, n_{q_1+q_2}$ by the divisor bound. Hence, the moment tends to zero when $N \rightarrow \infty$, and we have the same conclusion for $0 \leq q_2 < q_1$. Finally, if $0 \leq q_1 = q_2 = q$, by assumption, the moment is equivalent to $N^{-q}$ times the number of trivial solutions of the corresponding diophantine equation, i.e. to the corresponding moment for the sum of i.i.d. variables, uniform on the unit circle. By the central limit theorem, 
$$\mathbb{E}[|Y_N|^{2q} ] \underset{N \rightarrow \infty}{\longrightarrow} 
\mathbb{E}[|\mathcal{N}_{\mathbb{C}}|^{2q}].$$
We have then proven that for all integers $q_1, q_2 \geq 0$, 
$$\mathbb{E}[Y_N^{q_1} \overline{Y_N}^{\, q_2} ] \underset{N \rightarrow \infty}{\longrightarrow} 
\mathbb{E}[|\mathcal{N}_{\mathbb{C}} ^{q_1} \overline{\mathcal{N}_{\mathbb{C}}}^{\, q_2} ],$$
which gives the claim. 
\end{proof}
 We have proven the assumption of the previous proposition for $q \in \{1,2\}$, however, our method does not generalize to larger values of $q$. The divisor bound gives immediately a domination by $N^{q + o(1)}$ for the number of solutions, and then it seems reasonable to expect that the arithmetic constraints implied by the equation are sufficient to save at least a small power of $N$. Note that the situation is different for the sum $\sum_{n=1}^N 
 X_n$: for example, for $q = 2$, the number of non-trivial solutions of the equation $n_1 n_2 = n_3 n_4$ for $1 \leq n_1, n_2, n_3, n_4 \leq N$ is not $o(N^2)$, as we can see 
 by considering the equalities $a(2b) = b(2a)$ for $a$ and $b$ odd, $a < b$. 
 
 The previous proposition giving a "conditional CLT" can be generalized to the sums of the form
 $$\sum_{n=1}^N \prod_{j=1}^k X_{n + j}^{m_j},$$ 
 when the $m_j$'s have the same sign. The situation is more difficult if the $m_j$'s have different signs since the divisor bound alone does not directly give a useful bound on the number of solutions. 
\section{Convergence of the empirical measure in the case of roots of unity}  \label{uniform2q}
Here, we suppose that $(X_p)_{p \in \mathcal{P}}$ are i.i.d. uniform on the set 
$\mathbb{U}_q$ of $q$-th roots of unity, $q\geq 1$ being fixed. With the notation of the previous section, we now get: 
\begin{proposition} \label{qboundL2}
Let $m_1, \dots, m_k$ be integers, not all 
divisible by $q$, let $\epsilon > 0$ and let  $N > N' \geq 0$.  Then, 
$$\mathbb{E} \left[ 
\left|\sum_{n=N'+1}^N \prod_{j=1}^k X_{n+j}^{m_j} \right|^2 \right] 
\leq C_{q,k, \epsilon} (N-N') N^{\epsilon}$$
and 
$$\mathbb{E} \left[ |\hat{\mu}_{k,N}(m_1, \dots, m_N)|^2
\right] \leq \frac{C_{q,k, \epsilon}}{N^{1-\epsilon}},$$
where $C_{q,k, \epsilon} > 0$ depends only on $q, k, \epsilon$. 
\end{proposition}
\begin{proof} 
We can obviously assume that $m_1, \dots, m_k$ are between $0$ and $q-1$, which gives finitely many possibilities for these integers, depending only on $q$ and $k$. We can then suppose that $m_1, \dots, m_k$ are fixed at the beginning. 
We have to upper bound the number of couples $(n_1,n_2)$ on $\{N'+1, \dots, N\}^2$ such that 
$$\frac{\prod_{j=1}^k (n_1 +j)^{m_j}}{
\prod_{j=1}^k (n_2 +j)^{m_j}} \in 
(\mathbb{Q}_+^*)^q,$$
where, in this proof, $(\mathbb{Q}_+^*)^q$ denotes the set of $q$-th powers of positive rational numbers.
Now, any positive integer $r$  
can be decomposed as a product of a "smooth" integer whose prime factors are all strictly smaller than $k$,
 and a "rough" integer whose prime factors are all larger than or equal to $k$. If 
 the "rough" integer is denoted $\sharp_k(r)$, the condition just above implies: 
 $$\frac{\sharp_k \left(\prod_{j=1}^k (n_1 +j)^{m_j}\right)}{
\sharp_k \left(\prod_{j=1}^k (n_2 +j)^{m_j}\right)} \in 
(\mathbb{Q}_+^*)^q.$$
Now, the numerator and the denominator of this expression can both be written in a unique way as a product of a $q$-th perfect power and an integer whose $p$-adic valuation is between $0$ and $q-1$ for all $p \in \mathcal{P}$. If the quotient is a $q$-th power, necessarily the numerator and the denominator have the same "$q$-th power free" part. Hence, there exists a $q$-th power free integer $g$ such that 
$$\sharp_k \left(\prod_{j=1}^k (n_1 +j)^{m_j}\right), \;
\sharp_k \left(\prod_{j=1}^k (n_2 +j)^{m_j}\right) \in g \mathbb{N}^q,$$
$\mathbb{N}^q$ being the set of $q$-th powers of positive integers. 
Hence, the number of couples $(n_1,n_2)$ we have to estimate is bounded by 
$$\sum_{g \geq 1,  q\operatorname{-th \,  power \,  free}} [\mathcal{N}(q,k,g,N',N)]^2,
$$
where $\mathcal{N}(q,k,g,N',N)$ is the 
number of integers $n \in \{N'+1, \dots, N\}$ such that 
$$\sharp_k \left(\prod_{j=1}^k (n +j)^{m_j}\right) \in g \mathbb{N}^q.$$
If a prime number $p \in \mathcal{P}$ divides $n + j$ and $n+j'$ for $j \neq j' \in \{1,\dots, k\}$, it divides $|j-j'| 
\in \{1, \dots, k-1\}$, and then $p < k$. Hence, the rough parts of $(n+j)^{m_j}$ are 
pairwise coprime. Now, if $g_1, \dots, g_k$ are the $q$-th power free integers such that $\sharp_k[(n+j)^{m_j}] \in g_j \mathbb{N}^q$, we have $g_1 g_2 \dots g_k \in  g \mathbb{N}^q$. Now, $g_1, \dots, g_k$ are coprime, and then $g_1 g_2 \dots g_k$ is $q$-th power free, which implies $g_1 \dots g_k = g$. 
Hence 
$$\mathcal{N}(q,k,g,N',N)
\leq \sum_{g_1 g_2 \dots g_k = g}
\left| \left\{n \in \{N'+1, \dots N\}, 
\, \forall j \in \{1, \dots, k\}, \; 
 \sharp_{k}[(n+j)^{m_j}] \in g_j \mathbb{N}^q \right\} \right|. $$
 Let us now fix an index $j_0$ such that $m_{j_0}$ is not multiple of $q$. We have 
 $$\mathcal{N}(q,k,g,N',N)
\leq \sum_{g_1 g_2 \dots g_k = g}
\left| \left\{n \in \{N'+1, \dots N\}, 
\, \sharp_{k}[ (n +j_0)^{m_{j_0}}] \in g_{j_0} \mathbb{N}^q, \forall j \neq j_0, \; 
\operatorname{rad}(g_j)|(n+j) \right\} \right|,$$
where $\operatorname{rad}(g_j)$ denotes the product of the distinct prime factors of $g_j$. 
The condition on $(n +j_0)^{m_{j_0}}$
means that for all $p \in \mathcal{P}$, 
$p \geq k$, 
$$m_{j_0} v_p (n+j_0) \equiv v_p (g_{j_0})
 \, (\operatorname{mod. } q),$$
 i.e. $v_p (g_{j_0})$ is divisible by $\operatorname{gcd}( m_{j_0}, q)$ and 
 $$(m_{j_0}/ \operatorname{gcd}( m_{j_0}, q)) v_p (n+j_0) \equiv
 v_p (g_{j_0})/ \operatorname{gcd}( m_{j_0}, q)\, (\operatorname{mod. } \rho_{j_0}),$$
 where $\rho_{j_0} := q/ \operatorname{gcd}( m_{j_0}, q) $. Since $m_{j_0}/ \operatorname{gcd}( m_{j_0}, q)$ is coprime 
 with $\rho_{j_0}$, the last congruence is
 equivalent to a congruence modulo $\rho_{j_0}$ between $v_p(n+j_0)$ and a fixed integer, which is not divisible by 
 $\rho_{j_0}$ if and only if $p$ divides 
 $g_{j_0}$. We deduce that the 
 condition on $(n+j_0)^{m_{j_0}}$ implies that $\sharp_k (n+j_0) \in h(q,m_{j_0},g_{j_0}) \mathbb{N}^{\rho_{j_0}}$, i.e. 
 $$n+j_0 = \alpha h(q,m_{j_0},g_{j_0})
 A^{\rho_{j_0}},$$
 where $\alpha$ is a $\rho_{j_0}$-th power 
 free integer whose prime factors are strictly smaller than $k$, $A$ is an integer and $h(q,m_{j_0},g_{j_0})$ 
 is an integer depending only on $q$, $m_{j_0}$ and $g_{j_0}$, which is divisible by 
 $\operatorname{rad} (g_{j_0})$. 
For a fixed value of $\alpha$, the 
values of $A$ should be in the interval 
$$I = \left( \big((N' +j_0)/[\alpha h(q,m_{j_0}, g_{j_0})] \big)^{1/\rho_{j_0}},
\big((N +j_0)/[\alpha h(q,m_{j_0}, g_{j_0})] \big)^{1/\rho_{j_0}} \right],$$
whose size is at most 
$$ [\operatorname{rad}(g_{j_0})]^{-1/\rho_{j_0}} [(N+j_0)^{1/\rho_{j_0}} -
(N'+j_0)^{1/\rho_{j_0}} ]
\leq 1+ \left(\frac{N - N'}{\operatorname{rad}(g_{j_0})} \right)^{1/2},
 $$
 by the concavity of the power $1/\rho_{j_0}$, the fact that $\rho_{j_0}
  \geq 2$ since $m_{j_0}$ is not divisible by $q$, which implies 
$x^{1/\rho_{j_0}} \leq 1 + \sqrt{x}$.
Now, the conditions on $n+j$ for $j \neq j_0$ imply a condition of congruence 
for $\alpha h (q, m_{j_0}, g_{j_0}) A^{\rho_{j_0}}$, modulo all the primes dividing one of the $g_j$'s for $j \neq j_0$. These primes do not divide $\alpha$, since $\alpha$ has all prime factors smaller than $k$, and $g_j$ divides $\sharp_k[(n + j)^{m_j}]$. They also do not divide $h (q, m_{j_0}, g_{j_0})$, 
since this integer has the same prime factors as $g_{j_0}$, which is prime with $g_j$. 
Hence, we get a condition of congruence 
for $A^{\rho_{j_0}}$ modulo all primes dividing $g_j$ for some $j \neq j_0$. 
For each of these primes, this gives 
at most $\rho_{j_0} \leq q$ congruence classes for $A$, and then, by the chinese reminder theorem, we get at most $q^{\omega\left(\prod_{j \neq j_0}
g_j\right)}$ classes modulo 
$\prod_{j \neq j_0} \operatorname{rad}(g_j)$, where $\omega$ denotes the number of prime factors of an integer.  
 The number of integers $A \in I$ satisfying the congruence conditions is then at most: 
 $$q^{\omega\left(\prod_{j \neq j_0}
g_j\right)}
\left[ 1 + \frac{1}{\prod_{j \neq j_0}
\operatorname{rad}(g_j)} \left(
1 + \left(\frac{N-N'}{\operatorname{rad}(g_{j_0})}
\right)^{1/2} \right) \right]
\leq [\tau(g)]^{\log q/\log 2}
\left[ 2 + \left( \frac{N - N'}{\operatorname{rad}(g)} \right)^{1/2} \right],$$
where $\tau(g)$ denotes the number of divisors of $g$. 
Now, $\alpha$ has prime factors smaller than $k$ and $p$-adic valuations smaller than $q$, which certainly gives $\alpha \leq 
(k!)^q$. Hence, by considering all the possible values of $\alpha$, and all the possible $g_1, \dots, g_k$, which  should divide $g$, we deduce 
$$ \mathcal{N}(q,k,g,N',N)
\leq    (k!)^q 
 [\tau(g)]^{k + (\log q/\log 2)}
\left[2 + \left( \frac{N - N'}{\operatorname{rad}(g)} \right)^{1/2} \right].$$
If $\mathcal{N}(q,k,g,N',N) > 0$, we have necessarily $$g \leq \prod_{j=1}^k (N+j)^{m_j} \leq (N+k)^{kq}
\leq (1+k)^{kq} N^{kq}.$$ Using the divisor bound, we deduce that for all $\epsilon > 0$, there exists $C^{(1)}_{q,k, 
\epsilon}$ such that for all $g \leq 
(1+k)^{kq} N^{kq}$, 
$$2 (k!)^q 
 [\tau(g)]^{k + (\log q/\log 2)}
 \leq C^{(1)}_{q,k, 
\epsilon} N^{\epsilon},$$
and then 
$$\mathcal{N}(q,k,g,N',N)
\leq   C^{(1)}_{q,k, 
\epsilon} N^{\epsilon}
 \left[ 1 + \left( \frac{N - N'}{\operatorname{rad}(g)} \right)^{1/2} \right],$$
i.e.
$$\mathcal{N}(q,k,g,N',N)
- C^{(1)}_{q,k, 
\epsilon} N^{\epsilon}
 \leq  C^{(1)}_{q,k, 
\epsilon} N^{\epsilon} 
(N - N')^{1/2} (\operatorname{rad}(g))^{-1/2}$$
which implies 
$$\left(\mathcal{N}(q,k,g,N',N)
- C^{(1)}_{q,k, 
\epsilon} N^{\epsilon}
\right)_+  \leq  C^{(1)}_{q,k, 
\epsilon} N^{\epsilon} 
(N - N')^{1/2} (\operatorname{rad}(g))^{-1/2}$$
since the right-hand side is nonnegative. 
Summing the square of this bound for all possible $g$ gives 
$$\sum_{g \geq 1,  q\operatorname{-th \,  power \,  free}} \left(\mathcal{N}(q,k,g,N',N)
- C^{(1)}_{q,k, 
\epsilon} N^{\epsilon}
\right)^2_+
\leq \left(C^{(1)}_{q,k, 
\epsilon}\right)^2 N^{2 \epsilon} (N-N') 
\sum_{g \geq 1,   q\operatorname{-th \,  power \,  free}}
\frac{\mathds{1}_{g \leq (1+k)^{kq} N^{kq}}}{\operatorname{rad}(g)}.$$
Now, since all numbers up to 
$(1+k)^{kq}N^{kq}$ have prime factors smaller than this quantity, we deduce, using the multiplicativity of the radical: 
\begin{align*}\sum_{g \geq 1,   q\operatorname{-th \,  power \,  free}}
\frac{\mathds{1}_{g \leq (1+k)^{kq} N^{kq}}}{\operatorname{rad}(g)}
&  \leq \prod_{p \in \mathcal{P}, 
 p \leq (1+k)^{kq}N^{kq}} \left(
 \sum_{j=0}^{q-1} \frac{1}{\operatorname{rad} (p^j)} \right)
\\ &  \leq \prod_{p \in \mathcal{P}, 
 p \leq (1+k)^{kq}N^{kq}} \left(
1 + \frac{q-1}{p} \right) \, 
 \leq \prod_{p \in \mathcal{P}, 
 p \leq (1+k)^{kq}N^{kq}} \left(
1 - \frac{1}{p} \right)^{1-q}
 \end{align*}
 which, by Mertens' theorem, is smaller 
 than a constant, depending on $k$ and $q$, 
 times $\log^{q-1}(1+ N)$. 
 We deduce that there exists a constant 
 $C^{(2)}_{q,k,\epsilon} >0$, such that 
 $$\sum_{g \geq 1,  q\operatorname{-th \,  power \,  free}} \left(\mathcal{N}(q,k,g,N',N)
- C^{(1)}_{q,k, 
\epsilon} N^{\epsilon}
\right)^2_+
\leq C^{(2)}_{q,k,\epsilon}
N^{3\epsilon} (N-N').$$
Now,  it is clear that 
$$\sum_{g \geq 1,  q\operatorname{-th \,  power \,  free}}\mathcal{N}(q,k,g,N',N) 
= N'-N,$$
since this sum counts all the integers $n$
from $N'+1$ to $N$, regrouped in function 
of the $q$-th power free part of 
$\sharp_k\left( \prod_{j=1}^k (n+j)^{m_j}
\right)$. 
Using the inequality $x^2 \leq (x-a)_+^2 + 2ax$, available for all $a, x \geq 0$, we deduce 
$$\sum_{g \geq 1,  q\operatorname{-th \,  power \,  free}}[\mathcal{N}(q,k,g,N',N)]^2 
\leq   C^{(2)}_{q,k,\epsilon}
N^{3\epsilon} (N-N') + 
2 C^{(1)}_{q,k, 
\epsilon} N^{\epsilon} (N-N').$$
This result gives the first inequality of the proposition, for 
$$C_{q,k,\epsilon}
=  C^{(2)}_{q,k,\epsilon/3}
+ 2C^{(1)}_{q,k,\epsilon/3}.$$
The second inequality is obtained by taking $N' = 0$ and dividing by $N^2$. 
\end{proof}
\begin{corollary}
For all $(m_1, \dots, m_k) \in \mathbb{Z}^k$, $\hat{\mu}_{k,N}(m_1, \dots, m_k)$ converges 
in $L^2$, and then in probability, to the corresponding Fourier coefficient of the uniform distribution $\mu_{k,q}$ on $\mathbb{U}_q^k$. In other words, $\mu_{k,N}$ converges weakly in probability to $\mu_{k,q}$. 
\end{corollary}
We also have a strong law of large numbers.
\begin{proposition}
Almost surely, $\mu_{k,N}$ weakly converges to $\mu_{k,q}$. More precisely, for all $(t_1, \dots, t_k) \in (\mathbb{U}_q)^k$, the proportion of $n \leq N$ such that 
$(X_{n+1}, \dots, X_{n+k}) = (t_1, \dots, t_k)$ is almost surely $q^{-k} + O(N^{-1/2 + \epsilon})$ for all $\epsilon > 0$. 
\end{proposition}
\begin{proof}
By Lemma \ref{lemmaLLN} and Proposition 
\ref{qboundL2}, we deduce that almost surely, for all $\epsilon > 0$, 
$0 \leq m_1, \dots, m_k \leq q-1$, 
$(m_1, \dots, m_k) \neq (0,0, \dots, 0)$, 
$$\hat{\mu}_{k,n}(m_1, \dots, m_k) 
= O( N^{-1/2 + \epsilon}).$$
Since we have finitely many values of $m_1, \dots, m_k$, we can take the $O$ uniform in $m_1, \dots, m_k$. Then, by inverting discrete Fourier transform on $\mathbb{U}_q^k$, we deduce the claim. 
\end{proof}
\section{More general distributions on the unit circle} \label{general}
In this section, $(X_p)_{p \in \mathcal{P}}$ are i.i.d., with any distribution on the unit circle. We will study the empirical distribution of $(X_n)_{n \geq 1}$, but not of the patterns $(X_{n+1}, \dots, X_{n+k})_{n \geq 1}$ for $k \geq 2$. 
More precisely, the goal of the section is to prove a strong law of large numbers for $N^{-1} \sum_{n=1}^N \delta_{X_n}$ when $N$ goes to infinity. We will use the following result, due to Hal\'asz, Montgomery and Tenenbaum (see \cite{Hal68}, \cite{Hal71},  \cite{GS}, \cite{M}, \cite{Te} p. 343): 
\begin{proposition} \label{HMT}
Let $(Y_n)_{n \geq 1}$ be a multiplicative function such that $|Y_n| \leq 1$ for all $n \geq 1$. For $N \geq 3, T > 0$, we set 
$$M(N,T) := \underset{|\lambda| \leq 2T} {\min}
\sum_{p \in \mathcal{P}, p \leq N} 
\frac{1 - \Re (Y_p p^{-i\lambda})}{p}.$$
Then: 
$$\left|\frac{1}{N} \sum_{n=1}^N Y_n \right|
 \leq C \left[(1 + M(N,T)) e^{- M(N,T) } + T^{-1/2} \right],$$
where $C > 0$ is an absolute constant. 
\end{proposition}
From this result, we show the following: 
\begin{proposition}
Let $(Y_n)_{n \geq 1}$ be a random multiplicative function such that $(Y_p)_{p  \in \mathcal{P}}$ are i.i.d., with $\mathbb{P} [|Y_p| \leq 1] = 1$,  $\mathbb{P} [Y_p = 1] < 1$
and  $\mathbb{P} [Y_p = - 1] < 1$. Then, almost surely,
for all $c \in (0, 1 - |\mathbb{E}[\Re(Y_2)]|)$
$$\frac{1}{N} \sum_{n=1}^N  Y_n 
 = O((\log N)^{-c}).$$
\end{proposition}
\begin{proof}
First, we observe that for $1 < N' < N$ integers, $\lambda > 0$,
$$\sum_{p \in \mathcal{P}, N' < p \leq N} 
p^{-1 - i\lambda} = 
\int_{N'}^{N} \frac{d \theta(x)}{x^{1+i \lambda} \log x} = \left[ \frac{\theta(x)}{x^{1+i \lambda} \log x} \right]_{N'}^N 
+ \int_{N'}^{N} \left( \frac{(1+i \lambda)}{x^{2 + i \lambda} \log x}
+ \frac{1}{x^{1 + i \lambda} \, x \log^2 x} \right) \theta(x) dx,$$

where, by a classical refinement of the prime number theorem,  
$$\theta(x) := \sum_{p \in \mathcal{P}, 
p \leq x} \log p = x + O_A(x/\log^{A} x)$$
for all $A > 1$.  
The bracket is dominated by $1/\log (N')$,  the second part of the last integral is dominated by 
$$\int_{N'}^{\infty} \frac{dx}{x \log^2 x} = \int_{\log N'}^{\infty} \frac{dy}{y^2} = 1/\log (N'),$$
and the error term of the first part is dominated by $(1+\lambda)/ \log^{A} (N')$. Hence
$$\sum_{p \in \mathcal{P}, N' < p \leq N}  p^{-1-i\lambda}= I_{N',N, \lambda} +  O_A \left( \frac{1}{\log N'} + 
\frac{\lambda}{ \log^{A} N'} \right),$$
where
$$I_{N',N, \lambda}  = (1+i\lambda) \int_{N'}^{N} \frac{dx}{x^{1 + i \lambda} \log x} = 
(1+i\lambda) \int_{\lambda \log N'}^{ \lambda \log N} 
\frac{e^{-i  y}}{y} dy. 
$$
Now, for all $a \geq 1$, 
$$\int_{a}^{\infty} \frac{e^{-i  y}}{y} dy = \left[ \frac{e^{-i  y}}{-iy} \right]_a^{\infty} 
- \int_{a}^{\infty} \frac{e^{-i  y}}{iy^2} dy = O(1/a),$$
which gives 
$$I_{N',N, \lambda} 
=  \int_{\lambda \log N'}^{ \lambda \log N} 
\frac{e^{-i  y}}{y} dy 
+ O(1/\log N').$$
Now, the integral of $(\sin y)/y$ on $\mathbb{R}^*_+$ is  conditionally convergent: $(\sin y)/y$ tends to $1$ when $y \rightarrow 0$ and 
the convergence of the integral at $\infty$ is easily deduced from an integration by parts.
Hence, the integral of $(\sin y/y)$ on any interval of $\mathbb{R}^*_+$ is uniformly bounded, which implies 
$$\Im(I_{N',N,\lambda}) = O(1).$$
We deduce 
$$\Im \left(\sum_{ p \in \mathcal{P}, 
N' < p \leq N}   p^{-1-i\lambda} \right)
= O_A \left( 1 + \frac{\lambda}{\log^A N'} \right).$$
Bounding the sum on primes smaller than $N'$ by taking the absolute value, we get: 
$$\left|\Im \left(\sum_{ p \in \mathcal{P}, 
 p \leq N}
 p^{-1-i\lambda} \right) \right|
\leq   \log \log (3+N') + O_A \left( 1 +  \frac{\lambda}{\log^A N'} \right),$$
and then by taking $N' = e^{(\log N)^{10/A}}$, for $N$ large enough depending on $A$, 
$$\left| \Im  \left(\sum_{ p \in \mathcal{P}, 
 p \leq N}
 p^{-1-i\lambda} \right) \right|
 \leq \frac{ 10\log \log N}{A} + O_A \left( 1 + \frac{\lambda}{\log^{10} N} \right),
 $$
 $$\underset{N \rightarrow \infty}{\lim \sup}  \sup_{0 < \lambda \leq \log^{10} N}
(\log \log N)^{-1} \left|\Im \left(\sum_{ p \in \mathcal{P}, 
 p \leq N}
 p^{-1-i\lambda} \right) \right|
 \leq 10/A,$$
 and then by letting $A \rightarrow \infty$ and using the symmetry of the imaginary part for $\lambda \mapsto -\lambda$, 
 $$ \sup_{|\lambda| \leq \log^{10} N} \left|\Im \left(\sum_{ p \in \mathcal{P}, 
 p \leq N}
 p^{-1-i\lambda} \right) \right|
 = o(\log \log N)$$
 for $N \rightarrow \infty$. This estimate can also be deduced from known bounds on the Riemann zeta function on the line $\Re = 1 + 1/\log N$. 
 Now,  for all $\rho$ whose real part 
 is in $[-1,1)$, we have 
 \begin{align*}\min_{|\lambda| \leq \log^{10} N}  \sum_{p \in \mathcal{P}, p \leq N}
 \frac{1 - \Re(\rho \, p^{-i \lambda})}{p} 
 & \geq
  \min_{|\lambda| \leq \log^{10} N}  \sum_{p \in \mathcal{P}, p \leq N}
 \frac{1 - \Re(\rho) \Re( p^{-i \lambda})}{p} 
\\ & -  \max_{|\lambda| \leq \log^{10} N}   \left| \sum_{p \in \mathcal{P}, p \leq N}
 \frac{\Im(\rho) \Im( p^{-i \lambda})}{p}  \right|.
 \end{align*}
  The first term is at least the sum of 
  $1 - |\Re(\rho)|$ divided by $p$, and then at least $[1 - |\Re(\rho)| + o(1)] \log \log N$. The second term is $o(\log  \log N)$ by the previous discussion. Hence,
  \begin{equation}\min_{|\lambda| \leq \log^{10} N}  \sum_{p \in \mathcal{P}, p \leq N}
 \frac{1 - \Re(\rho \, p^{-i \lambda})}{p}  \geq [1 - |\Re(\rho)| + o(1)] \log \log N. \label{eqrho}
 \end{equation}
 Now, let $\rho := \mathbb{E}[Y_2]$, and 
 $Z_{p,\lambda}  := \Re[(Y_p - \rho)
 p^{-i\lambda}]$. The variables 
 $(Z_{p,\lambda})_{p \in \mathcal{P}}$ are centered, independent, 
  bounded by  $2$. By Hoeffding's lemma (see, for example, Massart \cite{Massart}, p. 21), 
 for all $u \geq 0$, 
 $$\mathbb{E}[e^{u Z_{p,\lambda}/p}]
 \leq e^{2(u/p)^2},$$
 and then by independence, 
 $$\mathbb{E}[e^{u \sum_{p \in \mathcal{P},  p \leq N} Z_{p,\lambda}/p}] \leq e^{2u^2 \sum_{p \in \mathcal{P}, p \leq N} p^{-2}} \leq e^{2u^2 (\pi^2/6)} \leq e^{4u^2},$$
\begin{align*} \mathbb{P} \left[
 \sum_{p \in \mathcal{P},  p \leq N} \frac{Z_{p,\lambda}}{p} 
 \geq (\log \log N)^{3/4} \right] 
& \leq e^{-(\log \log N)^{3/2}  / 8} 
 \mathbb{E} \left[ e^{(1/8)(\log \log N)^{3/4}  \sum_{p \in \mathcal{P},  p \leq N} \frac{Z_{p,\lambda}}{p} } \right] 
\\ & \leq e^{-(\log \log N)^{3/2}  / 8} 
e^{4 [(1/8)(\log \log N)^{3/4}]^2} 
= e^{- (\log \log N)^{3/2}  / 16}.
\end{align*} 
Applying the same inequality to $-Z_{p, \lambda}$, we deduce 
$$\mathbb{P} \left[ \left|
 \sum_{p \in \mathcal{P},  p \leq N} \frac{\Re[(Y_p- \rho)p^{-i\lambda}] }{p}  \right|
 \geq (\log \log N)^{3/4} \right] 
  \leq 2 e^{- (\log \log N)^{3/2}  / 16},$$
$$\mathbb{P} \left[ \max_{|\lambda| \leq 
\log^{10} N,  \lambda \in (\log^{-1} N) \mathbb{Z} }  \left|
 \sum_{p \in \mathcal{P},  p \leq N} \frac{\Re[(Y_p- \rho)p^{-i\lambda}] }{p}  \right|
 \geq (\log \log N)^{3/4} \right]
 = O\left( \log^{11}N \, e^{-  (\log \log N)^{3/2} /16} \right).$$
 The derivative of the last sum  
  in $p$ with respect to $\lambda$ is dominated by 
  $$ \sum_{p \in \mathcal{P}, p \leq N} 
 \frac{\log p}{p} = O(\log N)$$
 and then the sum cannot vary more than 
 $O(1)$ when $\lambda$ runs between two consecutive multiples of $\log^{-1} N$. 
 Hence, 
 \begin{align*}
 \mathbb{P} \left[ \max_{|\lambda| \leq 
\log^{10} N} \left|
 \sum_{p \in \mathcal{P},  p \leq N} \frac{\Re[(Y_p- \rho)p^{-i\lambda}] }{p}  \right|
 \geq (\log \log N)^{3/4} + O(1) \right] 
 &  = O \left( (\log N)^{11 - \sqrt{\log \log N}/16} \right)
 \\&  = O(\log^{-10} N).
 \end{align*}
 If we define, for $k \geq 1$, $N_k$ as the integer part of $e^{k^{1/5}}$, we deduce, by Borel-Cantelli lemma, that almost surely, for all but finitely many $k \geq 1$, 
 $$\max_{|\lambda| \leq 
\log^{10} N_k} \left|
 \sum_{p \in \mathcal{P},  p \leq N_k} \frac{\Re[(Y_p- \rho)p^{-i\lambda}] }{p}  \right| \leq (\log \log N_k)^{3/4} + O(1).$$
 If this event occurs, we deduce, using \eqref{eqrho}, 
 $$\min_{|\lambda| \leq \log^{10} N_k}  \sum_{p \in \mathcal{P}, p \leq N}
 \frac{1 - \Re(Y_p \, p^{-i \lambda})}{p}  \geq [1 - |\Re(\rho)| + o(1)] \log \log N_k.$$
 Then, by Proposition  \ref{HMT}, we get 
 $$\left|\frac{1}{N_k}  
 \sum_{n=1}^{N_k} Y_n \right| 
 \leq C \left[ \left( 1 + 
 [1 - |\Re(\rho)| + o(1)] \log \log N_k
 \right) (\log N_k)^{-(1 - |\Re(\rho)|) + o(1)}  + \sqrt{2} \, \log^{-5} N_k \right].$$
Since $- (1 - |\Re(\rho)|) \geq -1 > -5$, we deduce 
$$\left|\frac{1}{N_k}  
 \sum_{n=1}^{N_k} Y_n \right| = 
 O((\log N_k)^{-(1 - |\Re(\rho)|) + o(1)}),$$
 which gives the claimed result along the sequence $(N_k)_{k \geq 1}$. 
 Now, if $N \in [N_k, N_{k+1}]$, we have, since all the $Y_n$'s have modulus at most $1$, 
 \begin{align*} 
 \left|\frac{1}{N_k}  
 \sum_{n=1}^{N_k} Y_n 
 - \frac{1}{N}  \sum_{n=1}^{N} Y_n 
 \right|
 & \leq 
\left| \frac{1}{N}  \sum_{n=N_k+1}^N Y_n
\right| + \left( \frac{1}{N_k} - \frac{1}{N} \right)  \left|\sum_{n=1}^{N_k} Y_n \right|
\leq \frac{N - N_k}{N} + N_k \left(\frac{1}{N_k} - \frac{1}{N} \right)
\\ & = \frac{2(N - N_k)}{N}
 \leq  \frac{2 (e^{(k+1)^{1/5}} - e^{k^{1/5}} + 1)}{e^{k^{1/5}} - 1}
 = O \left(e^{(k+1)^{1/5} - k^{1/5}} - 1 + e^{-k^{1/5}} \right) 
 \\ & = O( k^{-4/5} ) = O(\log^{-4} N).
\end{align*}
This allows to remove the restriction to the sequence $(N_k)_{k \geq 1}$. 
\end{proof}
Using Fourier transform, we deduce a law of large numbers for the empirical measure $\mu_{N} = \frac{1}{N} \sum_{n=1}^N \delta_{X_n}$, under the assumptions of this section. 
\begin{proposition}
If for all integers $q \geq 1$, $\mathbb{P} [X_2 \in \mathbb{U}_q] < 1$, then almost surely, $\mu_N$ tends to the uniform measure on the unit circle. 
\end{proposition} 
\begin{proof}
For all $m \neq 0$, $X_2^m$ takes its values on the unit circle, and it is not a.s. equal to $1$. Applying the previous proposition to $Y_n = X_n^m$, we deduce that $\hat{\mu}_N(m)$ tends to zero almost surely, which gives the desired result. 
\end{proof}
\begin{proposition}
If for $q \geq 2$, $X_2 \in \mathbb{U}_q$ almost surely, but $\mathbb{P} [X_2 \in \mathbb{U}_r] < 1$ for all strict divisors $r$ of $q$, then almost surely, $\mu_N$ tends to the uniform measure on $\mathbb{U}_q$. More precisely, almost surely, for all $t \in \mathbb{U}_q$, the proportion of $n \leq N$ such that $X_n = t$ is $q^{-1} + O((\log N)^{-c})$, as soon as 
$$c < \inf_{1 \leq m \leq q-1}
\left(1 - \mathbb{E}[\Re(X_2^m)] \right),$$
this infimum being strictly positive. 
\end{proposition}
\begin{proof}
The infimum is strictly positive since by assumption, $\mathbb{P}[X_2^m = 1] < 1$ for all $m \in \{1, \dots, q-1\}$. 
Now, we apply the previous result to $Y_n = X_n^m$ for all $m \in \{1, \dots, q-1\}$, and we get the claim after doing a discrete Fourier inversion. 
\end{proof}
 

\begin{thebibliography}{99}
\bibitem{B} 
A. Baker: \emph{Bounds for solutions of hyperelliptic equations},
Proc. Cambridge Phil. Soc.
\textbf{65}
(1969), 439--444.
\bibitem{BJ}
 H. Bohr and B. Jessen:  \emph{\"Uber die Wertverteilung der Riemannschen Zetafunktion, Erste Mitteilung }, Acta Math. \textbf{54} (1930), 1--35, \emph{Zweite Mitteilung}, Acta Math. \textbf{58} (1932), 1--55.
\bibitem{Ch}
 S. Chowla: \emph{The Riemann hypothesis and Hilbert's tenth problem}, Gordon and Breach, New York, 1965.
\bibitem{D}
A. Dubickas: \emph{A note on the multiplicative dependence of consecutive integers}, Scient. works of Lith. Math. Soc.: suppl. to "Liet. Matem. Rink.", Technika, Vilnius (1998), 21--23.
\bibitem{GS}
A. Granville, K. Soundararajan: \emph{Decay of Mean Values of
Multiplicative Functions}, Canad. J. Math. \textbf{55} (2003), 1191--1230.
\bibitem{HP}
L. Hajdu, A. Pint\'er: \emph{Square product of three integers in short intervals}, Math. of Computation \textbf{68} (1999), 1299--1301. 
\bibitem{Hal68}
G. Hal\'asz: \emph{\"Uber die Mittelwerte multiplikativer zahlentheoretischer Funktionen},  Acta Math. Acad. Sci. Hung. \textbf{19} (1968), 365--403.
\bibitem{Hal71}
G. Hal\'asz: \emph{On the distribution of additive and the mean values of multiplicative arithmetic functions}, Studia Sci. Math. Hung. \textbf{6} (1971), 211--233. 
\bibitem{Hal}
G. Hal\'asz: \emph{On  random  multiplicative  functions}.  In
Hubert  Delange  Colloquium  (Orsay,
1982), Publications Math\'ematiques d'Orsay
\textbf{83}, 74--96. Univ. Paris XI, Orsay,  1983.
\bibitem{Harper}
A. Harper: \emph{Moments of random multiplicative functions, I: Low moments, better than squareroot cancellation, and critical multiplicative chaos}, Forum of 
Mathematics, Pi \textbf{8} (2020). 
\bibitem{Harper2}
A. Harper: \emph{Moments of random multiplicative functions, II: High moments}, Algebra Number Theory \textbf{13}, no. 10 (2019), 2277--2321. 
\bibitem{HNR}
A. Harper, A. Nikeghbali, M. Radziwi\l \l: 
\emph{A note on Helson's conjecture on moments of random multiplicative functions}, preprint. To appear in "Analytic Number Theory" in honor of Helmut Maier's 60th birthday. 
\bibitem{HL}
W. Heap, S. Lindqvist: \emph{Moments of random multiplicative functions and truncated characteristic polynomials}, 
preprint (2015). arXiv:1505.03378
\bibitem{H}
H. Helson: \emph{Hankel Forms},
Studia Math. \textbf{198} (2010), 79--84. 
\bibitem{Hil}
A. Hildebrand: \emph{On consecutive values of the Liouville function}, Enseign. Math. (2) \textbf{32} (1986), 219--226.
\bibitem{Jarvis} 
F. Jarvis: \emph{Algebraic Number Theory}. Springer, 2014. 
\bibitem{LTW}
Y.-K. Lau, G. Tenenbaum, J. Wu: \emph{On mean values of random multiplicative functions}, Proceedings of the Amer. Math. Soc. 
\textbf{141}, no. 2 (2013), 409--420. 
\bibitem{Massart}
P. Massart: \emph{Concentration Inequalities and Model Selection}. Ecole d'Et\'e de Probabilit\'es de Saint-Flour XXXIII. Springer, 2003. 
\bibitem{MRT}
K. Matom\"aki, M. Radziwill, T. Tao, \emph{Sign patterns of the Liouville and M\"obius functions}, preprint (2015). arXiv:1509.01545
\bibitem{M}
H.-L. Montgomery: \emph{A note on the mean values of multiplicative functions}, Inst. Mittag-Leffer, Report
\textbf{17} (1978).
\bibitem{Rosen}
K. Rosen: \emph{Elementary number theory and its applications}. Addison-Wesley Pub. Co., 1984. 
\bibitem{RS}
J. B. Rosser, L. Schoenfeld: \emph{Approximate formulas for some functions of prime numbers}, Illinois J. Math \textbf{6} (1962), 64--94. 
%\bibitem{R}
% J.H. Rickert: \emph{Simultaneous rational approximations and related diophantine equations}, Proc. Cambridge Philos. Soc.
%\textbf{113}
%(1993), 461-472.
\bibitem{Sh}
T.-N. Shorey: \emph{On 
linear forms in the logarithms of algebraic numbers}, Acta 
Arith. \textbf{30} (1976-77), 27--42. 
\bibitem{TT1}
T. Tao, J. Ter\"av\"ainen: \emph{The structure of logarithmically averaged correlations of multiplicative functions, with applications to the Chowla and Elliott conjectures}, Duke Math. Journal \textbf{168}, no. 11 (2019), 1977--2027. 
\bibitem{TT2}
T. Tao, J. Ter\"av\"ainen: \emph{Odd order cases of the logarithmically averaged Chowla conjecture}, Journal de Th\'eorie des Nombres de Bordeaux \textbf{30} (2017), 
997--1015.
\bibitem{TT3}
T. Tao, J. Ter\"av\"ainen: \emph{The structure of correlations of multiplicative functions at almost all scales, with applications to the Chowla and Elliott conjectures}, Algebra 
and Number Theory \textbf{13} (2019), 2103--2150. 
\bibitem{T}
J. Turk: \emph{Multiplicative properties of integers in short intervals}, Indag. Math. (Proceedings) \textbf{83} (1980), 429--436.
\bibitem{Te}
G. Tenenbaum, \emph{Introduction
 to analytic and probabilistic
number theory}.
Cambridge University Press,
1995. 
\bibitem{W}
A. Wintner, \emph{Random factorizations and Riemann hypothesis},
Duke Mathematical Journal \textbf{11} (1944), 267--275. 
 \end{thebibliography}
\end{document}